\documentclass{amsart}
\usepackage{amsmath,amsthm}
\usepackage{amsfonts,amssymb}

\usepackage{enumerate}

\newtheorem{thm}{Theorem}[section]
\newtheorem{cor}[thm]{Corollary}
\newtheorem{lem}[thm]{Lemma}
\newtheorem{prop}[thm]{Proposition}
\newtheorem{exam}[thm]{Example}

\newtheorem{defn}[thm]{Definition}

\theoremstyle{remark}
\newtheorem{rem}{Remark}[section]

 \def\a{{\alpha}} 
 \def\b{{\beta}}
 \def\g{{\gamma}}

 \def\l{{\lambda}}

 \def\s{{\sigma}}
 \def\la{{\langle}}
 \def\ra{{\rangle}}

 \def\CB{{\mathcal B}}
 \def\CA{{\mathcal A}}
 \def\CG{{\mathcal G}}

 \def\CP{{\mathcal P}}     
      
 \def\CT{{\mathcal T}}

 \def\NN{{\mathbb N}}

 \def\RR{{\mathbb R}}
  \def\SS{{\mathbb S}}
 \def\ZZ{{\mathbb Z}}

\newcommand{\wt}{\widetilde}
\newcommand{\wh}{\widehat}

\begin{document}

\title[Positive cubature rules]
{On positive cubature rules on the simplex and isometric embeddings}
\author{Masanori Sawa}
\address{Graduate School of Information Science\\ Nagoya University\\
  Chikusa-ku, Nagoya 464-8601.}\email{sawa@is.nagoya-u.ac.jp}
\author{Yuan Xu}
\address{Department of Mathematics\\ University of Oregon\\
    Eugene, Oregon 97403-1222.}\email{yuan@math.uoregon.edu}

\begin{abstract}
Positive cubature rules of degree $4$ and $5$ on the $d$-dimensional simplex are constructed 
and used to construct cubature rules of index $8$ or degree $9$ on the unit sphere. The  
latter ones lead to explicit isometric embedding among the classical Banach spaces. Among 
other things, our results include several explicit representations of $(x_1^2+ \ldots + x_d^2)^t$ 
in terms of linear forms of degree $2t$ with rational coefficients for $t = 4$ and $5$. 
\end{abstract}

\date{\today}
\thanks{The work was supported in part by NSF Grant DMS-1106113}

\maketitle

\section{Introduction}
\setcounter{equation}{0}

A cubature rule of degree $n$ on the $d$-dimensional simplex is a finite sum of  linear combination of function 
evaluations that approximates the integral on the simplex, which becomes exact if the integrand is any 
polynomial of degree at most $n$. We are interested in constructing  positive cubature rules on the simplex
$$
   T^d := \{x \in \RR^{d+1}: x_1 \ge 0, \ldots, x_d\ge 0, 1-x_1-\ldots-x_d \ge 0\},
$$ 
where positive means that the coefficients in the linear combination are all positive numbers. Such objects are traditionally 
studied in the context of numerical analysis. They are, however,  closely related, even equivalent under mild restriction, to 
cubature rules on the unit sphere $\SS^d$ of $\RR^{d+1}$, which are intimately connected 
to several other topics in mathematics, such as spherical design and isometric embedding of the classical finite
dimensional Banach  space \cite{GS, Lyubich-Vaserstein}.

Our interests in this topic originally stem from a search for cubature rules on high dimensional sphere in the literature, which
turned up only a couple of rules of degree 9. Since an invariant cubature rule of degree $9$ on the sphere is equivalent to 
a cubature rule of degree $4$ for the Chebyshev weight function on the simplex, we then looked for cubature rules on the simplex 
and found, to our surprise, no positive cubature rules of degree 4, albeit for the unit weight function, in higher dimensions 
\cite{DW, Stroud, St76}. It turns out that there is a reason for this omission: most of the construction in the literature relies on 
choosing the points in advance and the points being chosen, devoid of other reasons, possess simple structures, such as
centroid $(\frac{1}{d+1},\ldots, \frac{1}{d+1})$ and points in the form $(r,\ldots, r, 0, \ldots, 0)$ or in its orbit under the symmetric 
group of the simplex. However, such choices, as we shall prove in Section 3, do not lead to positive cubature rules of degree 4 or higher. 
Adding another point of a different nature, we then look for a practical way of constructing positive cubature rules. The most 
straightforward way of constructing a cubature rule of degree $n$ is to solve the moment equations that match the 
integral and the cubature sum for all polynomials up to degree $n$. A well known theorem of Sobolev  \cite{Sobolev} states 
that if the cubature rule is symmetric under certain finite group, then only polynomials invariant under the same group need 
to be considered. Restricting to invariant polynomials reduces the size of the equations considerably. The difficulty, however, 
lies in that the moment equations are nonlinear in the coordinates of the points. Such nonlinear equations can be solved mostly only by numerical means and for cubature rules of moderate degrees in lower dimensions; see, for example,
\cite{ZCL} for the latest effort in this direction. By 
adding only one point of the form $u_a = (a,b,\ldots,b,0,\ldots,0)$, we shall provide a method of constructing positive 
cubature rules without solving the equations numerically. 

The positive cubature rules that we found are of degree $4$ and $5$ on the simplex and degree 9 and 11 on the sphere,
from which cubature rules of the same degree but with far less number of points can be deduced from a method due to 
Victoir \cite{Victoir} that uses orthogonal arrays and block designs.  Even after the deduction, however, the number of nodes
for our cubature rules can still be fairly large comparing to non-positive cubature rules. If the positivity is not required, then 
there exist cubature rules of remarkably small number of points of higher degrees for all dimensions \cite{GM, HX}. For  
simplex and sphere of higher dimensions, the positivity exerts a high demand that is not easy to meet. Besides the numerical
stability that they automatically bring, the positive cubature rules are often useful in other context. The cubature rules of 
degree $2 t+1$ on the sphere  with central symmetry are equivalent to cubature rules of index $2t$ on the sphere, which give, when they are positive, 
isometric embeddings between classical Banach space $l_2^d \mapsto l_{2t}^N$. Moreover, some of our solutions are 
rational, meaning that the coefficients are all rational and $a$ is a square of an integer, which lead to representations of 
$(x_1^2+\ldots +x_d^2)^t$ as a sum of real linear forms of power $2t$ for $t = 4,5$ with rational coefficients. 
Such representations have implications in number theory. 
 
The paper is organized as follows. The next section is preliminary, where we sum up the background on cubature rules
and their connection to other topics. The non-existence of the positive cubature rules for simple choice of points and our 
method of constructing positive cubature rules are given in Section 3. The method is applied to integrals with respect to the constant weight function on the simplex in Section 4 and to Chebyshev weight function, which connects to cubature rules on the sphere, in Section 5.

\section{Preliminaries}
\setcounter{equation}{0}

\subsection{Cubature rules on the simplex and on the sphere}
For $x \in T^d$, define $|x|_1= x_1+\ldots + x_d$. For $\g > -1$ and $x \in T^d$, define 
$$
    W_\g(x) = \left( x_1 \ldots x_d (1- |x|_1) \right)^{\g}, \qquad x \in  T^d. 
$$ 
Let $w_\g$ be the normalization constant of $W_\g$ defined by 
$$
    w_\g = 1 \Big /  \int_{T^d} W_\g(x) dx  = \frac{\Gamma( (d+1)(\g+1) )}{[\Gamma(\g+1) ]^{d+1} }.
$$
Let $\Pi_n^d$ denote the space of polynomials of degree at most $n$ in $d$ variables. 
A cubature rule of degree $n$ for the weight function $W_\g$ is a finite sum of function 
evaluations such that 
\begin{equation} \label{cuba-Td}
       w_\g \int_{T^d} f(x) W_\g(x) dx  = \sum_{k=1}^N \l_k f(x_k), \quad \forall f \in \Pi_n^d, 
\end{equation}
and there is a function $f^* \in \Pi_{n+1}^d $ for which the above equation fails to hold.
We are particularly interested in cubature rules that have all nodes inside $T^d$, i.e.,
$x_k \in T^d$ and that are positive, meaning that all $\l_k >0$.  

The simplex $T^d$ is invariant under the permutations of its vertices. To capture the symmetry, it is often easier 
to work with homogeneous coordinates, $\xi \in \RR^{d+1}$ with $|\xi|_1 =1$, under which the simplex 
$T^d$ is equivalent to 
$$
     \CT^{d+1} := \{\xi \in \RR^{d+1}: \xi_1 \ge 0, \ldots, \xi_{d+1} \ge 0, |\xi|_1 =1\}.
$$
Under homogeneous coordinates, $W_\g(x)$ becomes $\xi^\g$ and we have 
\begin{equation} \label{Td-Td}
   \int_{\CT^{d+1}} f(\xi) \xi^\gamma d\xi = \int_{T^d} f(x, 1-|x|_1) W_\g(x) dx. 
\end{equation}
Working with $\CT^{d+1}$ allows us look at cubature rules of a particular form.  

\begin{defn} \label{defn:index}
Let $\CP_n^{d+1}$ denote the space of homogeneous polynomials of degree $n$ in $\RR^{d+1}$. 
A cubature rule is of index $n$ (in contrast to degree $n$) if it is exact for all polynomials in $\CP_n^{d+1}$. 
\end{defn}

\begin{prop}
\label{prop:cuba-CTd}
The cubature rule \eqref{cuba-Td} with all nodes inside $T^d$ is equivalent to
a cubature rule of index $n$ for the integral against $\xi^\g$ on $\CT^{d+1}$:
\begin{equation} \label{cuba-CTd}
    w_\g \int_{\CT^{d+1}} f(\xi) \xi^\g d\xi  = \sum_{k=1}^N \l_k f(\xi_k), \qquad \forall f \in \CP_n^{d+1}, 
\end{equation}
where $\xi_k$ and $x_k$ are related through $\xi_k=(x_k, 1-|x_k|_1) \in \CT^{d+1}$.
\end{prop}

\begin{proof}
Every
homogeneous polynomial $F \in \CP_n^{d+1}$ can be written as 
$$
    F(x_1,\ldots, x_{d+1})
= \sum_{j=0}^n f_j(x_1,\ldots,x_d) x_{d+1}^{n-j} \quad \hbox{with} \quad f_j \in \Pi_j^d, 
$$
which implies, in particular, that $f(x) := F(x_1,\ldots, x_d, 1-|x|_1) \in \Pi_n^d$.
Indeed, applying \eqref{cuba-Td} to such $f$ gives, by \eqref{Td-Td},
the cubature rule \eqref{cuba-CTd}. On the other hand,
for each $f \in \Pi_m^d$ with $0 \le m \le n$,
applying \eqref{cuba-CTd} to the homogeneous polynomials $f(x_1,\ldots,x_d)|x|_1^{n-m} \in \CP_n^{d+1}$
gives \eqref{cuba-Td}. 
\end{proof}

Switching from \eqref{cuba-Td} to \eqref{cuba-CTd} has one advantage, namely, that
we can disregard the restriction of
$\xi_k \in \CT^{d+1}$. Indeed, \eqref{cuba-CTd} is equivalent to 
\begin{equation*} 
    w_\g \int_{\CT^{d+1}} f(\xi) \xi^\g d\xi  = \sum_{k=1}^N \l_k^*  f(\xi_k/ r_k ), \qquad \forall f \in \CP_n^{d+1}, 
\end{equation*}
where $\lambda^*= \lambda_k r_k^n$ for any $r_k > 0$. Therefore we can always rescale so that the points 
are inside $\CT^{d+1}$. This fact is especially useful for constructing cubature rules of lower degree 
in high dimensions, as will be seen in the following sections. For $\alpha \in \NN_0^{d+1}$,
the moment $m_\a^{(\g)}$ with respect to $\xi^\g$ on $\CT^{d+1}$ is given by 
\begin{equation}\label{moment-g}
     m_\a^{(\g)}: = w_\g \int_{\CT^{d+1}} \xi^ {\a + \g} d \xi
  = \frac{ (\g+1)_{\a_1}\cdots (\g+1)_{\a_{d+1} }}
          {(( \g +1)(d+1))_{
|\a|_1 }},
\end{equation}
where $(a)_m = :a(a+1)\ldots (a+m-1)$.

We are interested in two integrals on the simplex, corresponding to $\gamma = 0$ and $\gamma = -1/2$, 
with respect to the weight functions 
$$
    W_0(x) =1 \quad \hbox{and}\quad  W_{-1/2}(x) = 1/\sqrt{x_1 \ldots x_d (1-|x|_1)},
$$
respectively. The first integral is traditionally studied in numerical analysis, whereas the second one
is of interest because a cubature rule for this weight function corresponds to a cubature rule for the 
surface measure $\s$ on the sphere $\SS^d$. In fact, the cubature rule \eqref{cuba-Td} of 
degree $n$ is equivalent to a cubature rule of degree $2n-1$ on the sphere,
$$
     \s_d  \int_{\SS^d} f(x) d \sigma = \sum_{k=1}^M \mu_k f(y_k), \quad \forall f \in \Pi_{2n-1} (\SS^d),
$$
where $\Pi_n(\SS^d)$ denotes the space of polynomials of degree at most $n$ in $d+1$ variables
and $\s_d  = w_{-1/2}/2$ is the surface area of $\SS^d$. 
To state this relation  \cite{X98} more precisely,  for  $x \in \RR^{d+1}$, let $\tau(x)$ denote the number 
of non-zero components of $x$. For $x= (x_1,\ldots, x_d) \in T^d$, let 
$$
     X = (x_1,\ldots, x_d, 1-|x|_1) \quad\hbox{and}\quad  \sqrt{X} = (\sqrt{x_1},\ldots, \sqrt{x_d}, \sqrt{1-|x|_1}). 
$$

\begin{prop}
\label{prop:cuba-SSd0}
If the cubature rule \eqref{cuba-Td} is of degree $n$ and it has all nodes $x_k \in T^d$, then 
\begin{equation*}
  \frac{w_\g}{2} \int_{\SS^d} f(x) \prod_{i=1}^d x_i^{2 \gamma +1} d\s(x) =  \sum_{k=1}^M 
      \frac{\lambda_k}{2^{\tau(u_k)}} \sum_{ g\in \ZZ_2^{d+1}} f( u_k g), \quad \forall f \in \Pi_{2n+1}(\SS^d), 
\end{equation*}
where $u_k = \sqrt{X_k}$. Furthermore, this relation is reversible. 
\end{prop}
 
A subset $X$ of $\RR^{d+1}$ is said to be {\it podal}, if for any distinct $x, x' \in X$, $x \ne -x'$.
We denote $-X = \{(-x_1, \ldots, -x_{d+1}) : (x_1, \ldots, x_{d+1}) \in X\}$. The set $X \cup (-X)$ is an 
{\it antipodal} subset of $\RR^{d+1}$. The following is a slight generalization of Proposition 4.3 of~\cite{Lyubich-Vaserstein},
which follows immediately from Propositions \ref{prop:cuba-CTd} and \ref{prop:cuba-SSd0}.

\begin{cor}
\label{cor:cuba-SSd1}
Let $\{x_1, \ldots, x_{N/2}\}$ be a podal subset of $\SS^d$.
If $x_k$ and $\lambda_k$ define
a cubature rule of index $2t$ on $\SS^d$, then 
\begin{equation} \label{cuba-SSd}
   \s_d \int_{\SS^d} f(x) d\s(x) =  \sum_{k=1}^{N/2} 
      \frac{\lambda_k}{2} \{f(x_k) + f(-x_k)\}, \quad \forall f \in \Pi_{2t+1}(\SS^d).
\end{equation}
Conversely,
any centrally symmetric cubature rule of degree $2t+1$ on $\SS^d$ with nodes $\pm x_k$
can be reduced to a cubature rule of index $2t$ on $\SS^d$ with nodes $x_k$.
\end{cor}

\subsection{Invariant cubature rules}

Let $\CG$ be a group of bijective linear transformations on $\RR^{d+1}$.
Let $\Omega$ be a $\CG$-invariant subset of $\RR^{d+1}$ and
$\mu$ be a $\CG$-invariant measure $\mu$ on $\Omega$.
For a point $x \in \RR^{d+1}$,
we denote the $\CG$-orbit of $x$ by $x \CG$.
A function $f$ on $\Omega$ is {\it $\CG$-invariant} if
for any $g \in \CG$ and $x \in \Omega$, $f(xg) = f(x)$.
A cubature rule with nodes $x_k$ and weights $\lambda_k$,
with respect to $\mu$,
is {\it $\CG$-invariant} if $x_k$ are partitioned into $\CG$-orbits, and
$x_k g = x_{k'}$ implies $\lambda_k = \lambda_{k'}$.

We are particularly interested in the group $\CA_d$
of permutations of the axes of $\RR^{d+1}$.
The group $\CA_d$ is the symmetric group $S_{d+1}$, and
an $\CA_d$-invariant function on $\RR^{d+1}$ is a 
symmetric function of $x_1,\ldots, x_{d+1}$.
We also deal with the Weyl group $\CB_{d+1}$ of type $B$.
This group contains a subgroup of reflections of the axes,
which is isomorphic to the elementary $2$-group $\ZZ_2^{d+1}$.
Throughout this paper we consistently use these notations.

The weight function $W_\g$ is invariant under any permutation of the $d+1$ vertices of the simplex $T^d$,
fixing the origin. This means that $T^d$ is $\CA_d$-invariant in the homogeneous coordinates of $\CT^{d+1}$.
A {\it symmetric cubature rule for the weight function $\xi^\g: = \xi_1^\g \cdots \xi_{d+1}^\g$} is of the form
\begin{equation}
\label{symmCuba}
\int_{\CT^{d+1}} f(\xi) \xi^\g d\xi = \sum_{k=1}^M \frac{\lambda_k}{| x_k \CA_d|}
\sum_{x \in  x_k\CA_d} f(x),
\end{equation} 
where $|x_k \CA_d|$ denotes the cardinality (the number of distinct elements) of the orbit $x_k \CA_d$.
It is obvious that the number of nodes of (\ref{symmCuba}) is $\sum_{k=1}^M |x_k\CA_d|$.
According to a well-known theorem of Sobolev~\cite{Sobolev},
a symmetric cubature rule of index $n$ holds if and only if
(\ref{symmCuba}) holds for every $f \in (\CP_n^{d+1})^{\CA_d}$, where $(\CP_n^{d+1})^{\CA_d}$ denotes
the space of homogeneous polynomials invariant under $\CA_d$.
The dimension of this space is given by the partition number
$$
 \dim (\CP_n^{d+1})^{\CA_d}
= \mathrm{card} \left \{ \b \in \NN^{d+1}: |\b| =d, \b_1 \ge \b_2\ge \cdots \ge \b_d \right\}.
$$
For small $n$ this means that we only need to verify a small number of invariant polynomials, which
can possibly be solved by working with the moment equations. 
A basis of the space  $(\CP_n^{d+1})^{\CA_d}$ is given by symmetric polynomials of degree $n$, which
are indexed by partitions. Let $\ell = (\ell_1,\ldots, \ell_r)$ be a partition, that is, $\ell_1 \ge \ell_2 \ge ...\ge \ell_r >0$, 
$\ell_i \in \NN$. We denote the symmetric homogeneous polynomial associated to $\ell$ by 
$$
  S_{\{ \ell \}}(x) = \mathrm{sym}(x_1^{\ell_1} \cdots x_r^{\ell_r}).  
$$

\subsection{Isometric embeddings of Banach spaces}

Let $p$ be a positive integer with $p \ne \infty$.
We define the norm $|\cdot |_p : \mathbb{R}^d \rightarrow \mathbb{R}$ by
\begin{equation*}
|x |_p = \bigg(\sum_{i=1}^d |x_i|^p \bigg)^{1/p}.
\end{equation*}
The Euclidean space $\mathbb{R}^d$ endowed with the norm $|\cdot |_p$
is a classical finite-dimensional Banach space, usually denoted by $l_p^d$.

Given two Banach spaces $l_p^d, l_{p'}^{d'}$, a classical problem in the theory of Banach spaces 
asks when there exists an $\mathbb{R}$-linear map $F : l_p^d \longrightarrow l_{p'}^{d'}$ such that
$$
        | F(x) |_{p'} = | x |_p \quad \text{for every ${\bf x} \in l_p^d$}.
$$
Such a map $F$ is called an {\it isometric embedding} from $l_p^d$ to $l_{p'}^{d'}$.
To exclude trivial cases, we assume $p' \ge p \ge 2$ and $p \ne p'$.
It is well known~\cite[Theorem 1.1]{Lyubich-Vaserstein} that if $p, p' \ne \infty$ and an 
isometric embedding from $l_p^d$ to $l_{p'}^{d'}$ exists, then $p = 2$ and $q$ is an even integer.
Hereafter we only consider the case $p = 2$, $p' = 2t < \infty$.

\begin{thm}
\label{thm:LV93}
The following are equivalent.
\begin{enumerate}
\item[(i)] There exists a cubature rule of index $2t$ on $S^{d-1}$ with $N$ points;
\item[(ii)] There exists an isometric embedding $l_2^d \longrightarrow l_{2t}^N$;
\item[(iii)] There exist $N$ vectors $r_1, \cdots, r_N \in \mathbb{R}^d$ such that
for any $x \in \mathbb{R}^d$,
\begin{equation}
\label{eq:iso}
\sum_{i=1}^N \langle x, r_i \rangle^{2t}
= \langle x, x \rangle^{ t},
\end{equation}
where $\langle \cdot, \cdot \rangle$ denotes the usual inner product.
\end{enumerate}
\end{thm}

This theorem was proved by Lyubich and Vaserstein~\cite{Lyubich-Vaserstein}. 
The equivalence ^^ ^^ (i) $\Leftrightarrow$ (iii)" was also observed by Reznick~\cite{Reznick}
in connection with Waring's problem in number theory.

Given a cubature rule of index $2t$ on $\SS^{d-1}$ with $N$ nodes $x_k$ and weights $\lambda_k$,
we can explicitly construct an isometric embedding $l_2^d \longrightarrow l_{2t}^N$: 
For $f(y) = \langle x, y \rangle^{2t} \in \CP_{2t}^d$,
\begin{equation}
\label{eq:iso2}
\sum_{i=1}^N \lambda_i f(x_i)
= \s_{d-1} \int_{\SS^{d-1}} |\langle x, y \rangle|^{2t} d\s(y) = c_t \langle x, x \rangle^t,
\end{equation}
where
\begin{equation*}
c_t :=\s_{d-1} \int_{\SS^{d-1}} y_1^{2t} d\s(y), \quad y  = (y_1, \ldots, y_d).
\end{equation*}
The mapping
\begin{equation*}
F(x) = (\sqrt[2t]{\lambda_1 / c_t} \langle x, x_1 \rangle, \ldots,
\sqrt[2t]{\lambda_N / c_t} \langle x, x_N \rangle)
\end{equation*}
defines an isometric embedding $l_2^d \longrightarrow l_{2t}^N$. (\ref{eq:iso2}) can be transformed into 
the form \eqref{eq:iso} by setting $r_i  = \sqrt[2t]{\lambda_i/c_t} x_i$.

The value of $c_t$ can be easily evaluated (it goes back to Hilbert~\cite{Hilbert}), 
\begin{equation*}
c_t = \frac{\displaystyle \Gamma ( \tfrac{d}{2}) \Gamma (t + \tfrac12) } {\displaystyle  \Gamma (t+\tfrac{d}{2}) \Gamma (\tfrac 12 )}
     = \frac{(2t-1)!!(d-2)!!}{(d+2t-2)!!},
\end{equation*}
which is of course a rational number. Thus, for example, if there is a $\CB_d$-invariant cubature rule of index $2t$ on $\SS^{d-1}$
such that the weights $\lambda_k$ are rational and the nodes $x_k$ are of the form
$(\tfrac{b_k}{\sqrt{a_k}}, \ldots, \tfrac{b_k}{\sqrt{a_k}}, \tfrac{c_k}{\sqrt{a_k}}, \ldots, \tfrac{c_k}{\sqrt{a_k}})$
with $a_k, b_k, c_k$ being rational, then (\ref{eq:iso2}) shows that 
$$
 c_t (x_1^2+\ldots + x_d^2)^t = \sum_{k=1}^N \frac{\l_k}{a_k^{t}} \la x, ( b_k , \ldots,  b_k ,  c_k , \ldots,  c_k ) \ra^{2 t},
$$
which gives a representation of $\la x, x \ra^t$ as a sum of the $2t$ powers of linear forms with only 
rational coefficients. Such representations are of interest in number theory; see~\cite{Dickson, Reznick} for details.

In Section 5 we will construct some examples of index eight cubature rules on the sphere whose
nodes and weights involve only rational numbers.

\subsection{Block designs and orthogonal arrays}

Victoir~\cite{Victoir} proposed a method that can be used to reduce the size of invariant cubature rules 
on the simplex and on the sphere. His approach utilizes two combinatorial objects, namely $t$-designs and orthogonal arrays. 
We first explain these two concepts.

Let $t, k, v, \lambda$ be integers such that $0 \le t \le k \le v$.
A {\it $t$-design} is a system of a set $V$ of $v$ elements, called {\it points},
and a collection $\CB$ of $k$-element subsets of $V$, called {\it blocks},
such that every $t$-element subset of $V$ appears in exactly $\lambda$ blocks of $\CB$.
This is denoted by $t$-$(v, k, \lambda)$. It is well known (cf.~\cite{GKL07}) that
if a $t$-$(v, k, \lambda)$ with $b$ blocks exists, then
\begin{equation}
\label{eq:Tdes}
b = \lambda \frac{\binom{v}{t}}{\binom{k}{t}}
= \l \frac{v(v-1) \cdots (v-t+1)}{k(k-1) \cdots (k-t+1)}.
\end{equation}
An incidence matrix of a $t$-$(v, k, \lambda)$ is
a $v \times b$ zero-one matrix with rows and columns being indexed by points and blocks, respectively,
such that for $x \in V, B \in \CB$, its $(x, B)$th entry is $1$ iff $x \in B$.

Let $s, l, L, \mu$ be nonnegative integers. An {\it orthogonal array with strength $s$, constraints $l$ 
and index $\mu$}, is an $L \times l$ matrix, such that in every $s$ columns, each of the $2^s$ ordered 
$s$-tuples of elements $\pm1$ appears exactly $\mu$ times among $L$ rows. This is often denoted by
$OA(L, l, 2, s)$; we do not put $\mu$ in the notation, because $\mu = L/2^s$, by the definition of OA.

We need a bit more notations. Let $e_1 = (1,0,\ldots,0), \ldots, e_{d+1} = (0,\ldots,0,1)$
be the standard basis vectors of $\RR^{d+1}$. Define 
$$
v_k^{(\alpha,\beta)}
= \alpha \sum_{i=1}^k e_i + \beta \sum_{i=k+1}^{d+1} e_i. 
$$
We recall the notation $\tau(\cdot)$ as in Subsection 2.1.
Victoir proved the following.

\begin{prop}
\label{prop:cuba-Victoir}
\begin{enumerate}[i]
\item[(i)] 
Assume there exists a symmetric cubature rule of index $t$
on $\CT^{d+1}$, such that
\begin{align*}
 w_\gamma \int_{\CT^{d+1}} f(\xi) \xi^\gamma d\xi    =  \sum_{i=1}^\ell \frac{\lambda_i}{\binom{d+1}{k_i}}
     \sum_{x \in
v_{k_i}^{(\alpha_i, \beta_i)}
\CA_d} f(x)  
  + \sum_{i=\ell+1}^M \frac{\lambda_i}{| x_i \CA_d|} \sum_{x \in x_i \CA_d} f(x)
\end{align*}
Moreover assume that for $i = 1, \ldots, \ell$,
there exists a $t$-design with $d+1$ points and $b_i$ blocks of size $k_i$, with
an incidence matrix $I_i$.
Let $X_i$ be the columns of
$\alpha_i I_i + (\beta_i - \alpha_i) J_{d+1,b_i}$, where
$J_{d+1,b_i}$ denotes the all-one matrix of size $(d+1) \times b_i$.
Then,
\begin{align*}
  w_\gamma \int_{\CT^{d+1}} f(\xi) \xi^\gamma d\xi  
 =  \sum_{i=1}^\ell \frac{\lambda_i}{b_i} \sum_{x \in X_i} f(x) +
\sum_{i=\ell+1}^M \frac{\lambda_i}{|x_i \CA_d |} \sum_{x \in x_i \CA_d} f(x)
\end{align*}
defines a cubature rule of index $t$ on $\CT^{d+1}$.
\item[(ii)] 
Assume there exists a $\ZZ_2^{d+1}$-invariant
cubature rule of degree $2t+1$ on the sphere $\SS^d$:
\begin{equation*}
 \sigma_d \int_{\SS^d} f(x) d\s(x) =
\sum_{i=1}^M \frac{\lambda_i}{2^{\tau(x_i)} }
   \sum_{x \in  x_i\ZZ_2^{d+1} } f(x).
\end{equation*}
Moreover assume there exist $OA(|X_i|, \tau(x_i), 2, 2t+1)$ with rows $X_i$, $i=1, \ldots, M$.
Then,
\begin{equation*}
 \sigma_d  \int_{\SS^d} f(x) d\s(x) =
\sum_{i=1}^M \frac{\lambda_i}{|X_i|} \sum_{x \in X_i} f(x)
\end{equation*}
defines a cubature rule of degree $2t+1$ on $\SS^d$.
\end{enumerate}
\end{prop}

Let us briefly explain on how Proposition~\ref{prop:cuba-Victoir} is used; see~\cite[Subsection 4.4]{Victoir}.
We start with a symmetric cubature rule of index $t$ 
on $\CT^{d+1}$. By Proposition~\ref{prop:cuba-Victoir}~(i), if there is a $t$-design for 
the set $\{v_{k_i}^{(\a_i,\b_i)}: i=1,\ldots, \ell\}$,
we can reduce the size of this cubature rule,
where (\ref{eq:Tdes}) is used to calculate the values of $b_i$ with $v = d+1$. 
By Propositions~\ref{prop:cuba-SSd0} and~\ref{prop:cuba-CTd}, the resulted cubature rule 
can be transformed to a $\ZZ_2^{d+1}$-invariant cubature rule of degree $2t+1$ on $\SS^d$.
Finally, by Proposition~\ref{prop:cuba-Victoir}~(ii), we reduce the size of the last cubature rule 
without reducing its degree of exactness. 

We emphasize that the process mentioned above is not exactly the same as that of Victoir,
who wrote his method in terms of cubature rules of degree-type. We slightly modify his approach 
to fit the cubature rules of index type, where Proposition~\ref{prop:cuba-CTd} serves as a bridge between
cubature of degree-type and cubature of index-type. In Sections 4 and 5, we will apply the method
to derive several new cubature rules on the simplex and on the sphere. 

\section{Symmetric cubature rules for the simplex}
\setcounter{equation}{0}

In the first subsection, we prove a result that explains why many of the cubature rules in \cite{Stroud} are
not positive. We then present our method of constructing cubature rules in the following three subsections.

\subsection{Positive cubature rules}
As mentioned in the introduction, many known cubature rules on $T^d$ are invariant and constructed 
by choosing their nodes among the centroid 
$  
    ( \tfrac{1}{\sqrt{d+1}}, \ldots, \tfrac{1}{\sqrt{d+1}})
$ 
and orbits of the points of the form \allowbreak $(r,\ldots, r, 0, \ldots, 0) \in T^d$.  Let us define these points
more precisely.  For $1 \le i \le d$, define $v_k = e_1+\ldots + e_k = (1, \ldots, 1, 0, \ldots, 0)$. Then the 
points are 
\begin{equation} \label{symmPoints}
        ( \tfrac{1}{\sqrt{d+1}}, \ldots, \tfrac{1}{\sqrt{d+1}}) \quad \hbox{and} \quad
        (r_k v_k )\CA_d  , \quad  1 \le k \le d +1,
\end{equation}
where $r_k$ are real numbers such that $r_k v_k = (r_k,\ldots, r_k, 0,\ldots, 0)\in T^d$, and
$(r_k v_k){\CA_d}$ means the orbit of $(r_k v_k, 1-|r_k v_k|)$ under the symmetric group of $d+1$ elements. 
 
\begin{thm} \label{thm:3.1}
\label{thm:cuba-Main1}
If a symmetric cubature rule of degree $n$ with $n \ge 4$ for $W_\g$ on $T^d$ uses only points form those in 
\eqref{symmPoints}, then it cannot be positive. 
\end{thm} 

\begin{proof}
The key ingredient of the proof is an orthogonal polynomial of degree $4$.
By Proposition \ref{prop:cuba-CTd},
we can consider a homogeneous orthogonal polynomial, call it $P_4$,
of degree $4$ on $\CT^{d+1}$. To derive this polynomial,
we use a relation between orthogonal polynomials on the simplex and
those on the sphere \cite{X98}, which shows that
$Y(x) : = P_4(x_1^2,\ldots,x_{d+1}^2)$ is an orthogonal polynomial on $\SS^d$
with respect to the measure $\prod_{i=1}^{d+1}|x_i|^{2 \g+1} d\sigma$.
The polynomial $Y$ then satisfies the equation $\Delta_h Y =0$, where 
$\Delta_h $ is the Dunkl Laplacian associated with the group
$\ZZ_2^{d+1}$.
The operator $\Delta_h$ is usually a 
differential-difference operator, but it becomes a differential operator when
it is applied to polynomials that are invariant
under $\ZZ_2^{d+1}$ (see, \cite[p. 156]{DX}). Consequently, the polynomial $Y$ is a solution of 
$$
\Delta_h Y =  \sum_{i=1}^{d+1} \frac{\partial^2 Y }{\partial x_i^2}
+ (2 \g+1) \sum_{k=1}^{d+1} \frac{1}{x_i}\frac{\partial Y}{\partial x_i} =0.
$$

With this explicit expression of $\Delta_h$, a straightforward computation shows that 
\begin{align*}
  \Delta_h S_{ \{8\}} (x) & = 16 (\g+4) S_{\{6\}}(x), \\
  \Delta_h S_{ \{6,2\} } (x) & = 12 (\g+3) S_{\{4,2\}}(x) + 4 d (\g+1) S_{\{6\}}(x),\\
  \Delta_h S_{ \{4,4\}} (x) & = 8 (\g+2) S_{\{4,2\}}(x),
\end{align*}
from which it is easy to verify that 
$$
   \Delta_h \left[ S_{ \{8\}} + a S_{ \{6,2\} } + bS_{ \{4,4\}} \right] =0, \quad
       a =  - \frac{4 (\g+4)}{d(\g+1)}, \quad 
       b =   \frac{6 (\g+3) (\g+4)}{ d (\g+1) (\g+2)}.
$$
Consequently, as discussed above, the polynomial 
$$
    P_4   =  S_{ \{4\}} + a S_{ \{3,1\} } +  b S_{ \{2,2\}}
$$
is an orthogonal polynomial of degree $4$ with respect to $\xi^\g $ on $\CT^d$.  

A quick computation using Lemma \ref{S-values} below shows that, for $1 \le k \le d$,  
\begin{align*}
 \frac{1}{| v_k \CA_d|} \sum_{x \in v_k \CA_d} P_4(x)
&
= P_4(v_k) = k + 2 \binom{k}{2} a +\binom{k}{2}b \\
  &  = \frac{  k ( d(\g+1) (\g+2) +  (\g+4)(\g +17) (k-1) )}{ d (\g+1) (\g+2)}   > 0.    
\end{align*}
Consequently, should an invariant cubature rule of index $4$ that uses only the orbits of $v_k$ be positive, 
we would have
$$
  0  = w_\g \int_{\CT^{d+1}} P_4(\xi) \xi^\g d\xi
= \sum_{k=1}^M
\frac{\l_k}{|v_k \CA_d|} \sum_{x \in v_k \CA_d } P_4(x) > 0,
$$
which is an obvious contradiction. Since a cubature rule of index $n$, $n > 4$, is automatically a cubature
rule of degree 4 as seen by Proposition \ref{prop:cuba-CTd}, this completes the proof. 
\end{proof}

Theorem~\ref{thm:cuba-Main1} explains why there are no positive cubature rules of degree $4$ or
higher in higher dimensions among those given in \cite{Stroud}. 
When $\gamma = -1/2$, this  result was proved by Bajnok~\cite{Bajnok}, where he worked with 
the cubature rules on the sphere. 


Accordingly, in order to construct positive cubature rules,
we need to include at least one point not among those 
in \eqref{symmPoints}. We shall add one point in the form of 
$$
      a e_1 + b e_2 + \ldots + b e_{p+1} = (a, b,\ldots, b, 0,\ldots, 0), \quad a, b > 0,
$$
which has $p$ elements of $b$ and one $a$. By Proposition \ref{prop:cuba-CTd},  we work with homogeneous 
polynomials for the integral over $\CT^{d+1}$, for which we can rescale of the point and consider 
$(a, 1,\ldots,1, 0,\ldots, 0)$ with $a > 0$ and $a \ne 1$. 

Let $k_1, \ldots, k_q$ be four distinct integers,
$1 \le k_1, \ldots, k_q \le d+1$.
We consider cubature rules of index $t$ 
in the form
\begin{equation} \label{CFdegreet}
   w_\g \int_{\CT^{d+1}} f(\xi) \xi^\g d\xi
      = \sum_{j=1}^q
\frac{A_j}{\binom{d+1}{k_j}} \sum_{x \in v_{k_j}\CA_d} f(x)
+
\frac{A_{q+1}}{d+1} \sum_{x \in u_{a,p}  \CA_d}  f(x),  
\end{equation}
where $u_{a,p}: = a e_1+ e_2 + \ldots +e_{p+1} = (a, 1,\ldots, 1, 0,\ldots, 0)$ with $a > 0$ and $a \ne 1$. 
Since the cubature rule is of index type, it can be rewritten, upon rescale, as
\begin{align} \label{CF-Inside}
   w_\g \int_{\CT^{d+1}} f(\xi) \xi^\g d\xi  =
& \sum_{j=1}^q
\frac{\l_j}{\binom{d+1}{k_j}}
      \sum_{x \in v_{k_j} \CA_d } f \left( \frac{x}{k_j} \right)  \\
& +
\frac{\lambda_{q+1}}{d+1}
\sum_{x \in  u_{a,p}\CA_d }   f \left(  \frac{x}{a+p }\right),  \notag
\end{align}
which has all nodes on $\CT^{d+1}$ and has  the weights given by 
\begin{equation} \label{lambda-A}
    \lambda_j = A_j k_j^t, \quad j =1, \ldots, q, \quad \hbox{and} \quad \lambda_{q+1} = A_{q+1} (a+p)^t. 
\end{equation}
By Proposition \ref{prop:cuba-CTd}, the cubature rule \eqref{CF-Inside} is equivalent to a cubature 
rule of degree $t$ on $T^d$. 

\subsection{Cubature rule of degree 4}  
We first describe our method of constructing positive cubature rules for rules of degree $4$ on $T^d$. 
In this case we choose $q =4$, the reason will be clear by the end of this subsection. 

Since the cubature rule \eqref{CFdegreet} is invariant under the group $\CA_d$, we only need to 
work with symmetric polynomials. There are five symmetric homogeneous polynomials of degree $4$
in $\RR^{d+1}$ for $d \ge 3$, which are 
$$
    S_{\{4\}}, \quad S_{\{3,1\}}, \quad S_{\{2,2\}}, \quad  S_{\{2,1,1\}}, \quad S_{\{1,1,1,1\}}. 
$$
We will need the values of these polynomials on the nodes of the cubature rule \eqref{CFdegreet}
which come down to a straightforward counting and the results are collected in the lemma below. 

\begin{lem} \label{S-values}
Let $\ell = (\ell_1, \ldots, \ell_r)$ be a partition of $t$ such that
$\ell_{i_j + 1} = \ldots = \ell_{i_{j+1}}, \ell_{i_j} \ne \ell_{i_{j+1}}$, where $j= 0, ..., s-1$.
For $1 \le k \le d+1$ and $u_{a,p} = (a,1,\ldots, 1,0,\ldots, 0)$, we have 
\begin{align*}
   S_{\{\ell\}}(v_k) & = \binom{k}{r} \binom{r}{i_1-i_0, i_2-i_1, \cdots, i_s-i_{s-1}}, \\
   S_{\{\ell\}}(u_{a,p}) & =
\sum_{j=0}^{s-1} a^{\ell_{i_{j+1}}} \binom{r-1}{i_1-i_0, \cdots, i_{j+1}-i_j-1, \cdots, i_s-i_{s-1}} \binom{p}{r-1} \\
& \qquad + \binom{r}{i_1-i_0, i_2-i_1, \cdots, i_s-i_{s-1}} \binom{p}{r}
\end{align*}
where $\binom{r}{i_1-i_0, i_2-i_1, \ldots, i_s-i_{s-1}} = \frac{r!}{(i_1-i_0)! (i_2-i_1)! \cdots (i_s-i_{s-1})!}$.
\end{lem}

We now apply the cubature rule \eqref{CFdegreet} to the symmetric polynomials. Let 
$m_{\{\ell\}}$ denote the integral of $ S_{\{\ell\}}$ on $\CT^{d+1}$, 
$$
     m_{\{\ell\}} := w_\g \int_{\CT^{d+1}} S_{\{\ell \}}(\xi) \xi^\g d\xi. 
$$ 
Their values can be easily evaluated using the explicit formula of
the moments of $\xi^\g d \xi$;
see the following section.

\begin{prop}
A positive cubature rule \eqref{CFdegreet} of index $4$ exists if 
\begin{equation} \label{A5}
            A_5 = \left( m_{\{3,1\} } -  2 m_{\{2,2\}} \right) / (a (a-1)^2 p ),
\end{equation}
and, for distinct positive integers $k_1,k_2,k_3,k_4$, the following system of equations 
\begin{align} \label{system}
\begin{split}
     \sum_{j=1}^4  A_j  k_j  + A_5 (a^4+p) & = m_{\{4\}} \\
     \sum_{j=1}^4  A_j  \binom{k_j}{2}  + A_5 \left(a^2 p + \binom{p}{2}\right) & = m_{\{2,2\}} \\
  3 \sum_{j=1}^4  A_j  \binom{k_j}{3}  + A_5 \left( (a^2   +2a) \binom{p}{2} + 3 \binom{p}{3}\right) & = m_{\{2,1,1\}} \\
     \sum_{j=1}^4  A_j  \binom{k_j}{4}  + A_5 \left( a  \binom{p}{3}+   \binom{p}{4} \right) & = m_{\{1,1,1,1\}}
\end{split} 
\end{align}
can be solved with positive $A_j$ and $a$. Furthermore, if $\{A_1,A_2, A_3, A_4, a\}$ is such a solution, 
then the positive cubature rule of index $4$ with all nodes in $\CT^{d+1}$ takes the form \eqref{CF-Inside} with
weights given by \eqref{lambda-A}. 
\end{prop}

\begin{proof}
By Sobolev' theorem, the existence of the cubature rule \eqref{CFdegreet}
is equivalent to the solvability of the system of equations
$$
       \sum_{j=1}^4 A_jS_{\{\ell \} } (v_{k_j}) + A_5 S_{\{\ell \}}(u_{a,p})  = m_{\{\ell\}}
$$
for all partitions of $4$. The equations in \eqref{system} correspond to
$\ell = \{4\}$, $\ell = \{2,2\}$, $\ell = \{2,1,1\}$ and $\ell = \{1,1,1,1\}$ by 
Lemma \ref{S-values}. The equation for
$\ell = \{3,1\}$
takes the form
$$
      2 \sum_{j=1}^4 A_j \binom{k_j}{2} + A_5 \left( (a^3+a) p + 2 \binom{p}{2}\right)   = m_{\{3,1\}}. 
$$
Taking the difference of this equation with 2 times of the equation for $\l = (2,2)$, we obtain 
$$
          a (a-1)^2 p A_5   = m_{\{3,1\} } -  2 m_{\{2,2\}},  
$$
which gives the formula for $A_5$. 
\end{proof}

We should comment on the nature of this proposition. In order to find cubature rules that use small number of 
points, we should solve the moment equations with $a$, in $u_{a,p}$, as a variable, which means that we should 
solve, heuristically, the moment equations with five parameters, say $A_1,A_2,A_3,A_4, a$. However, the equations 
in these parameters are nonlinear and can only be solved numerically. Instead, we have added one more 
coefficient and solve the linear system of equations for $A_1,A_2,A_3,A_4,A_5$ symbolically with $a$ as a 
free parameter. We then look for positive cubature rules by choosing $a$ judiciously. Whenever possible, 
we look for solutions for which $\l_i$ are all nonnegative and with as many being zero as possible to reduce 
the number of nodes. In fact, in many of our examples in the following section,
we have at least one zero coefficient.
For those, in essence, we have solutions for the non-linear system of five equations with five parameters. 

\subsection{Cubature of degree 5}
With the same set of nodes, we now consider cubature rules of degree $5$. It turns out that they are easier to
deal with. There are seven symmetric homogeneous polynomials of degree $5$ in $\RR^{d+1}$, which are 
$$
    S_{\{5\}},  \quad S_{\{4,1\}}, \quad   S_{\{3,1,1\}}, \quad  S_{\{3.2\}},
     \quad S_{\{2,2,1\}},\quad         S_{\{2,1,1,1\}}, \quad S_{\{1,1,1,1,1\}}. 
$$
Their values on the potential nodes are summarized in Lemma 3.2. 

It turns out that we can solve $a$ explicitly and turn the nonlinear system equations of moments into 
linear system. In other words, the seven moment equations can be solved with 7 parameters, $a, A_1,\ldots, A_6$.
Hence, we choose $q = 6$ in \eqref{CFdegreet}.

\begin{prop}
A positive cubature rule \eqref{CFdegreet} of index $5$ exists if $q = 6$, 
\begin{equation} \label{aA6}
            a =   \frac{p-1}{2} \frac{ m_{\{4,1\}} - m_{\{3,2\}} }{m_{\{3,1,1\}}- m_{\{2,2,1\}}} -1, \qquad 
            A_6 = \frac{m_{\{4,1\}} - m_{\{3,2\} }} {p a (a-1)^2(a+1)},
\end{equation}
and, for distinct positive integers $k_1,k_2,k_3,k_4, k_5$, the following system of equations 
\begin{align} \label{system5}
\begin{split}
   \sum_{j=1}^5  A_j  k_j  + A_6 (a^5+p) & = m_{\{5\}} \\
  2 \sum_{j=1}^5  A_j  \binom{k_j}{2}  + A_6 \left( a^4 p +a p + 2 \binom{p}{2}\right) & = m_{\{4,1\}} \\
  3 \sum_{j=1}^5  A_j  \binom{k_j}{3}  + A_6 \left( (a^3 +2a) \binom{p}{2} + 3 \binom{p}{3}\right) & = m_{\{3,1,1\}} \\
  4  \sum_{j=1}^5  A_j  \binom{k_j}{4}  + A_6 \left( (a^2 +3a ) \binom{p}{3}+ 4 \binom{p}{4} \right) & = m_{\{2,1,1,1\}} \\
    \sum_{j=1}^5  A_j  \binom{k_j}{4}  + A_6 \left( a  \binom{p}{4}+   \binom{p}{5} \right) & = m_{\{1,1,1,1,1\}} 
\end{split} 
\end{align}
can be solved with positive $A_j$. Furthermore, if $\{A_1,A_2, A_3, A_4, A_5\}$ is such a solution, 
then the positive cubature rule of index $5$ with all nodes in $\CT^{d+1}$ takes the form \eqref{CF-Inside} with
weights given by \eqref{lambda-A}. 
\end{prop}

\begin{proof}
Again, by Sobolev' theorem, we need only to solve the moment equaitons
$$
       \sum_{j=1}^4 A_jS_{\{\ell \} } (v_{k_j}) + A_5 S_{\{\ell \}}(u_{a,p})  = m_{\{\ell\}}
$$
for all partitions of $5$. The equations in 
\eqref{system5} correspond to 
$\ell = \{5\}$, $\ell = \{4,1\}$, $\ell = \{3,1,1\}$,
$\ell = \{2,1,1,1\}$ and $\ell = \{1,1,1,1,1\}$ 
by Lemma \ref{S-values}. The equations  for  $\ell = \{3,2\}$ and $\ell = \{2,2,1\}$
take the form
\begin{align*}
      2 \sum_{j=1}^5 A_j \binom{k_j}{2} + A_6 \left( (a^3+a^2) p + 2 \binom{p}{2}\right)  & = m_{\{3,2\}}, \\
     3 \sum_{j=1}^5  A_j  \binom{k_j}{3}  + A_6 \left( (2 a^2 + a) \binom{p}{2} + 3 \binom{p}{3}\right) & = m_{\{2,2,1\}}.
\end{align*}
Taking the differences of the equations for the pair $m_{\{4,1\}}$, $m_{\{3,2\}}$ and the equations for the pair 
of $m_{\{3,1,1\}}$ and $m_{\{2,2,1\}}$, we see that  
$$
     p a (a-1)^2(a+1) A_6  = m_{\{4,1\}} - m_{\{3,2\}}, \quad 
     \binom{p}{2} a (a-1)^2 A_6 = m_{\{3,1,1\}}- m_{\{2,2,1\}},
$$
from which we can solve for $a$ and $A_6$. 
\end{proof}

According to the proposition, the value of $a$ is determined before hand so that the
remaining moment equations become
a linear system, which can be easily solved. The fact that we no longer have the freedom of choosing the value of $a$ 
reduces our chance of finding positive cubature rules, but it does make our job of constructing cubature rules easier. 

\subsection{Cubature rule of degree 6 and above}

In this subsection, we show that the cases $t = 4, 5$ are the most that
we can do with the cubature rule \eqref{CFdegreet}.

\begin{prop}
\label{prop:nonexist}
If a cubature rule \eqref{CFdegreet} of index $6$ exists, then
\begin{equation} \label{eq:nonexist2}
   a = -\frac{2( 3m_{\{2,2,2\}} - m_{\{3,2,1\}} + m_{\{4,1,1\}})}{m_{\{3,2,1\}} - 6 m_{\{2,2,2\}}}
\end{equation}
and it has to satisfy the quadratic equation
\begin{align}  \label{eq:nonexist1}
(m_{\{4,2\}} - m_{\{3,3\}}) a^2 + (2 m_{\{4,2\}} - 2 m_{\{3,3\}} - m_{\{5,1\}}) a
         + m_{\{4,2\}} - 2 m_{\{3,3\}} =0.
\end{align}
\end{prop}

\begin{proof}
We substitute $S_{\{5,1\}}(x), S_{\{4,2\}}(x), S_{\{3,3\}}(x)$ into \eqref{CFdegreet}.
Then by Lemma~\ref{S-values},
\begin{align*}
\begin{split}
    2 \sum_{j=1}^q  A_j  \binom{k_j}{2}  + A_{q+1} \left( (a^5 + a) p + 2 \binom{p}{2} \right) & = m_{\{5,1\}} \\
    2 \sum_{j=1}^q  A_j  \binom{k_j}{2}  + A_{q+1} \left( (a^4 + a^2) p + 2 \binom{p}{2} \right) & = m_{\{4,2\}} \\
      \sum_{j=1}^q  A_j  \binom{k_j}{2}  + A_{q+1} \left( a^3 p + \binom{p}{2} \right) & = m_{\{3,3\}}
\end{split} 
\end{align*}
By multiplying both sides of the third equation by $2$ and
subtracting this new equation from the first equation, we obtain
\begin{equation}
\label{A5-1}
p A_{q+1} a (a-1)^2 (a+1)^2 = m_{\{5,1\}} - 2 m_{\{3,3\}}.
\end{equation}
Doing the same thing for the second and third equations, we have
\begin{equation}
\label{A5-2}
p A_{q+1} a (a-1)^2 a = m_{\{4,2\}} - 2 m_{\{3,3\}}.
\end{equation}
\eqref{eq:nonexist1} follows by \eqref{A5-1} and \eqref{A5-2}.
Next we shall show \eqref{eq:nonexist2}.
We substitute $S_{\{4,1,1\}}(x), S_{\{3,2,1\}}(x), S_{\{2,2,2\}}(x)$ into \eqref{CFdegreet}.
Then by Lemma~\ref{S-values},
\begin{align*}
\begin{split}
    3 \sum_{j=1}^q  A_j  \binom{k_j}{3}  + A_{q+1} \left( (a^4 + 2a) \binom{p}{2} + 3 \binom{p}{3} \right) & = m_{\{4,1,1\}} \\
    6 \sum_{j=1}^q  A_j  \binom{k_j}{3}  + A_{q+1} \left( 2(a^3 + a^2 + a) \binom{p}{2} + 6 \binom{p}{3} \right) & = m_{\{3,2,1\}} \\
      \sum_{j=1}^q  A_j  \binom{k_j}{3}  + A_{q+1} \left( a^2 \binom{p}{2} + \binom{p}{3} \right) & = m_{\{2,2,2\}}
\end{split} 
\end{align*}
By multiplying both sides of the first equation by $2$, and then
subtracting this equation from the second equation, we have
\begin{equation}
\label{A5-3}
-2 A_{q+1} a (a-1)^2 (a+1) \binom{p}{2} = m_{\{3,2,1\}} - 2 m_{\{4,1,1\}}.
\end{equation}
As for the second and third equations,
a similar consideration shows
\begin{equation}
\label{A5-4}
2 A_{q+1} a (a-1)^2 \binom{p}{2} = m_{\{3,2,1\}} - 6 m_{\{2,2,2\}}.
\end{equation}
\eqref{eq:nonexist1} thus follows by \eqref{A5-3} and \eqref{A5-4}.
\end{proof}

In the two cases that we are interested in, $\g = 0$ and $\g = -1/2$, Proposition \ref{prop:nonexist} can be
used to show the nonexistence of cubature rules of index $6$ on $\CT^{d+1}$. Indeed, in the case $\g = 0$, 
\eqref{eq:nonexist1} becomes, as can be verified using \eqref{moment-g}, $a^2 + a + 1 = 0$, for which there is 
no real solution. In the case of $\g = -1/2$,  \eqref{eq:nonexist2} shows that $a = 7/3$ and 
\eqref{eq:nonexist1} becomes $a^2 - 6 a + 1 = 0$, which does not have $a = 7/3$ as a solution. 
We sum this up in the following theorem. 

\begin{thm}
For $\g = 0$ and $-1/2$,
no matter what $q$ and $a$ are, there does not exist a cubature rule of the form
\eqref{CFdegreet} of index 6 or more. 
\end{thm}

\section{Positive  cubature rules of degree $4$ and $5$ on the simplex}
\setcounter{equation}{0}
In this section, we consider cubature rules of degree $4$ for the integral 
\begin{equation} \label{LebesgueT}
\frac{1}{d!} \int_{\CT^{d+1}} f(\xi)  d\xi = \frac{1}{d!} \int_{T^d} f(x)  dx.
\end{equation}
The moments for this integral is given by \eqref{moment-g} with $\g =(0,\ldots, 0)$, 
\begin{equation}\label{moment-0}
  m_\a^{(0)} = \frac{\a_1!\ldots \a_d!}{(d+1)_{
|\a|_1
}}, \qquad  \a \in \NN_0^d.  
\end{equation}
\subsection{Cubature rules of degree 4}
The integrals of $S_{\{\ell\}}$ for partitions of 4 are given by 
\begin{align*}
   m_{\{4\}} & = \frac{24}{(d+1)_4} (d+1),   \quad
   m_{\{2,2\}} = \frac{4}{(d+1)_4} \binom{d+1}{2}, \\ 
   m_{\{2,1,1\}} & = \frac{2}{(d+1)_4} \binom{d+1}{3}, \quad
   m_{\{1,1,1,1\}} = \frac{1}{(d+1)_4} \binom{d+1}{4}. 
\end{align*}
In order to have smaller number of points for a fixed dimension, we shall choose $p = d$, that is, 
$u_{a,p} = (a, 1,\ldots, 1)$. 

We report our results in cases depending on the values of $(k_1,k_2,k_3,k_4)$.
In order to have cubature rules with smaller numbers of nodes,
we shall choose $k_1 = d+1$ whenever possible and choose $k_1,k_2,k_3,k_4$
as small as possible otherwise.
Most of our positive cubature rules are of the type $(d+1,1,2,m)$ (except in Subsection 4.1.3), 
which are of the form 
\begin{align*}
  & \frac{1}{d!} \int_{\CT^{d+1}} f(\xi) d\xi = 
       \l_1 f
\left( ( \tfrac{1}{d+1}, \ldots, \tfrac{1}{d+1} ) \right)
      +
\frac{\l_2}{d+1}
\sum_{x \in  \left( 1,0,\ldots,0 \right)\CA_d} f(x)  \\
      & \qquad +
\frac{\l_3}{\binom{d+1}{2} }
\sum_{x \in \left(  \tfrac{1}{2}, \tfrac{1}{2},0,\ldots,0 \right)\CA_d} f(x) 
      +
\frac{\l_4}{\binom{d+1}{m}}
\sum_{x \in  \left(  \tfrac{1}{m}, \ldots,\tfrac{1}{m},0,\ldots,0  \right)\CA_d} f(x)  \\
     &  \qquad +
\frac{\l_5}{d+1}
\sum_{x \in  \left( \tfrac{a}{a+d}, \tfrac{1}{a+d}, \ldots, \tfrac{1}{a+d}  \right)\CA_d} f(x).  
\end{align*}
In presenting our examples, we will specify the value of $a$ but do not give the values of $\l_i$ in general, since 
these values can be obtained by solving the linear system of equations \eqref{system} once $a$ is known. 
We discuss the first case in some details to show how we arrive at our positive cubature rules.  

\subsubsection{Case $(k_1,k_2,k_3,k_4) = (d+1,1,2,3)$}
In this case, the system \eqref{system} has the solution
\begin{align*}
& A_1 =  \frac{6 - 7 a - 2 a^2 + a^3 - 2 d}{a (a -1)^2 (d+1)B(d)}, \quad A_3  =   \frac{d (2 - (a-2)(d-4))}{2 a B(d)}, \\
&  A_2  =    \frac{ -4 a^2 + a (38 - 7 d + d^2) - 2 (6 - 5 d + d^2)}{2 a B(d)},  \quad A_4  =  \frac{(a-d)d (d-1)}{6 a B(d)}
\end{align*}
where $B(d) = (d+2)(d+3)(d+4)$, with $A_5$ given in \eqref{A5}. 
For $A_4$ to be positive, we need $a > 2$. The positivity of $A_3$ holds if $d =3 ,4$ 
and for $a < 2 + 2/(d-4)$ if $d > 4$.
The nominator of $A_1$ is decreasing in $d$ and
it is not difficult to check that it is negative 
if $d =6$ and $2 < a < 2 +2/(6-4) = 3$. Thus,
we can have positive solutions only in the case of $d = 3, 4, 5$. 
We give in each case the values of $a$ and $\lambda_j$,
where $\l_j$ and $\wh A_j$ are related by, see \eqref{lambda-A}, 
$$
\l_1  = (d+1)^4 A_1, \quad \l_2  = A_2, \quad \l_3  = 2^4 A_3, \quad 
\l_4  = 3^4 A_4, \quad \l_5  = (a+d)^4 A_5, 
$$
so that $\sum_{j=1}^5 \l_j =1$. The value of $a$ that gives positive cubature rules is not unique. 
Whenever we can, we choose $a$ so that one of the coefficient is zero and
the resulting cubature rule has the smallest number of 
points among all choices of $a$. 

\medskip
{\it Case $d =3, 4$}. The cubature rules have
19 and 31  
nodes, respectively, with
$$  
   a = \frac{1}{8} (38 - 7 d + d^2 + \sqrt{1252 - 372 d + 93 d^2 - 14 d^3 + d^4}).
$$
For $d \ne 3, 4$, these $a$ and $\l_i$ give cubature rules of degree $4$ with non-positive coefficients with
$N = \binom{d+1}{3}+\binom{d+1}{2} +
2d+3 $ nodes. 

\medskip
{\it Case $d =5$}. The cubature rule has 32 nodes. 
\begin{align*}
   a  = 4, \quad   \l_1  = 0, \quad
\l_2 = \frac{1}{112},\quad \l_3  = 0, \quad \l_4 = \frac{15}{56} \quad   \l_5 = \frac{81}{112}.
\end{align*}

\subsubsection{Case $(k_1,k_2,k_3,k_4) = (d+1,1,2,m)$,  $m =4, 5, ...$}.
We will not give detail on how the solutions were found. To find the positive cubature 
rules, we consider the following two classes: one with $\l_2 =0$ and the other with $\l_3 =0$. 
It turns out that positive cubature rules occur only in a few cases in each class.
We obtain positive cubature rules only 
when $3 \le d \le 12$ and $d = 14, 15, 17$.   

\medskip\noindent
{\bf Class $\l_2 =0$}.
Positive cubature rules for $d = 4, 5, 6, 8, 9, 11,12, 14, 17$.
For each $d$, the cubature rule has at most
$N_{\CT^{d+1}}
= \binom{d+1}{m} + \binom{d+1}{2} +  d + 2$ nodes. 
Its coefficients are given as in the above with $\l_2 = 0$ and 
\begin{align*}
   a = & \frac{-16 - d + d^2 + 18 m -  2 d m + \sqrt{\Delta}} {4 (-1 + m)} \quad  \hbox{with} \\
    & \quad \Delta =-16 (-2 + d) (d - m) (-1 + m) + (16 + d - d^2 +  2 (-9 + d) m)^2.  
\end{align*}
It is easy to see that $a > 0$. We obtain positive cubature rules in the following cases: 
$$
  (d,m) = (4,4 ), (5 ,4 ), (6, 4), (8,4), (9, 5), (11,6 ), (12,6), (14, 7), (17, 8). 
$$

The case $d  = 4$ is noteworthy, since its number of nodes, $20$,
is less than the one in 4.1.1; moreover, $\l_1 =0$ and all 
coefficients are rational numbers, as seen by

\medskip
{\it Case $d =4$, $m=4$}. The cubature rule has 20 points.
\begin{align*}
   a  = 6, \quad   \l_1  = 0, \quad \l_2 = 0,\quad
\l_3  =  \frac{2}{21}, \quad \l_4 = \frac{32}{63} \quad   \l_5 = \frac{25}{63}.
\end{align*}

Another case that is noteworthy is $d =12$, which has both $\l_2$ and $\l_3$ equal to zero. 

\medskip
{\it Case $d =12$, $m=6$}. The cubature rule has 1731 points.
\begin{align*}
   a  = 6, \quad  \l_1= \frac{2197}{12250}, \quad \l_2  = 0, \quad  \l_3=0,  \quad  \l_4  = \frac{99}{245}, 
   \quad   \l_5  = \frac{729}{1750}.
 \end{align*}
 
\medskip\noindent
 {\bf Class $\l_3 =0$}. Positive cubature rules for $d =  7, 10, 12, 15$. For each $d$, the 
cubature rule has at most
$N_{\CT^{d+1}}
= \binom{d+1}{m} +  2 d + 3$ nodes. Its coefficient $\l_j$ are given 
with 
\begin{align*}
   a =  \frac{ 2(d-m)} {d+2 - 2m}, 
\end{align*}
which is positive if $m < (d+2)/2$. We obtain positive cubature rules for the parameters: 
$$
          (d,m) = (7, 4),  (10, 5), (12, 6), (15, 6).
$$

\subsubsection{Case $(k_1,k_2,k_3,k_4) = (d+1,1,3,m)$}
Positive cubature rules occur for several pairs of $(d,m)$. We can again consider two classes,
$\l_2 =0$ or $\l_3 =0$.
In the case of $\l_2 =0$, we obtain positive cubature rules for 
$$
    (d,m ) = (8,5), (9,5), (11,5), (12,6), (13,7), (14,8), (16, 8).
$$  
In the case of $\l_3 =0$, we obtain positive cubature rules for 
$$
    (d,m ) = (10,5), (12,6), (15,8).
$$  
In addition, for the parameters $(d, m) = (17,8)$,
we found a positive cubature rule with $\l_1 =0$. We shall not give these
rules explicitly. 

\medskip

Combining the above cases, we have found positive cubature rules of degree $4$
for $d = 3, 4, \ldots, 17$. On the other hand,
we have tried various of other combinations of $(k_1,k_2,k_3,k_4)$, but
found no positive cubature rules for $d > 17$.    

Although our main goal is to construct positive cubature rules,
it should be noted that the cubature rule for the case $(k_1,k_2,k_3,k_4) = (d+1,1,2,d)$,
non-positive as it is,
has the number of nodes in the order of $d^2 /2$, which is the 
lowest among all choices of $(k_1,k_2,k_3,k_4)$. Furthermore,
the case can be solved so that $\l_3 =0$ with a choice of
$a =2$, which is the case discussed in \cite{DW}.

\subsection{Cubature rules of degree 5}
We use Proposition 3.4 to find cubature rules of degree 5 for the integral \eqref{LebesgueT}. 
By \eqref{aA6} with $\g = 0$, we see that 
$$
  a =  \frac{6 (p - 1)}{d - 1} - 1 \quad \hbox{and} \quad  A_6 = \frac{12 d}{pa(a - 1)^2 (a + 1)}.
$$ 
With these parameters, the equation \eqref{system5} can always be solved and yield cubature
rules of degree 5 on the simplex. There are, however, only a few rules that are positive. These 
occur at the dimension $d = 3, 4, \ldots, 11$. All of our positive cubature rules are of the type 
$(d+1,1,2,m)$, which are of the form 
\begin{align*}
  & \frac{1}{d!} \int_{\CT^{d+1}} f(\xi) d\xi = 
       \l_1 f
\left( ( \tfrac{1}{d+1}, \ldots, \tfrac{1}{d+1} ) \right)
       + \frac{\l_2 }{d+1}
\sum_{x \in  \left( 1,0,\ldots,0 \right)\CA_d} f(x)  \\
      & \quad +   \frac{ \l_3}{\binom{d+1}{2} }
\sum_{x \in \left(  \tfrac{1}{2}, \tfrac{1}{2},0,\ldots,0 \right)\CA_d} f(x) 
      +    \frac{\l_4}{\binom{d+1}{3}}
\sum_{x \in
\left( \tfrac{1}{3}, \tfrac{1}{3}, \tfrac{1}{3}, 0,\ldots,0  \right)
 \CA_d} f(x)  \\
     &  \quad +    \frac{\l_5}{\binom{d+1}{m}}
\sum_{x \in  \left(  \tfrac{1}{m}, \ldots,\tfrac{1}{m},0,\ldots,0  \right)\CA_d} f(x)  + \frac{ \l_6}{(p+1)\binom{d}{p}}
\sum_{x \in  \left( \tfrac{a}{a+p},  \tfrac{1}{a+p}, \ldots, \tfrac{1}{a+p}, 0,\ldots,0  \right)\CA_d} f(x),  
\end{align*}
where $\lambda_i$ are normalized weight so that $\sum_{i=1}^6 \l_i =1$, which are given by
$$
\l_1  = (d+1)^5 A_1, \, \l_2  = A_2, \, \l_3  = 2^5 A_3, \, 
\l_4  = 3^5 A_4, \, \l_5  = m^5 A_5, \, \l_6  = (a+p)^5 A_5. 
$$
The cubature rule can be easily translated to a cubature rule on $T^d$ by 
Proposition 2.2. 
The parameters for our positive cubature rules of degree 5 are given by 
\begin{align*}
 (d,p,m) = & (3,3,4), (4,4,4), (5,4,4), (6,5,4), (7,6,5), (8,6,5),\\
               & (9,7,5), (10,8,5), (11,11,6), 
\end{align*}
and the number of nodes of the cubature rule is given by 
$$
  N = 1+ (d+1) + \binom{d+1}{2} + \binom{d+1}{3} + \binom{d+1}{m} + (p+1) \binom{d}{p}.
$$
It turns out that all solutions are rational. The cases $d =3$ and $d =11$ are notable as
each contains zero coefficients, which are given below: 

\medskip \noindent
{\it Case (d,p,m) = (3,3,4)}: 19 points with $a = 5$ and  
$$
   \l_1 = \frac{16}{105},  \,   \l_2  = \frac{1}{70},  \,   \l_3 = \frac{4}{35},  \,  
    \l_4 = \frac{81}{350},  \,   \l_5 = 0,  \,   \l_6 = \frac{256}{525}.
$$

\medskip \noindent
{\it Case (d,p,m) = (11,11,6)}: 1014 points with $a = 5$ and 
$$
   \l_1 =0, \,  \l_2 = \frac{1}{6825},  \,   \l_3 = \frac{11}{1365}, \,  
    \l_4 = 0,  \,   \l_5 = \frac{891}{2275},  \,   \l_6 = \frac{4096}{6825}.
$$

\subsection{Cubature rules with fewer nodes}
In this subsection we shall apply
Victoir's method to cubature rules on the simplex which
are obtained in Subsections 4.1 and 4.2.

We give three tables to list up the resulting cubature rules; the first two for cubature of 
index $4$ and the third one for cubature of index $5$. In each table,
the values of $d,m$ in the first column are the same notations
as in the previous two subsections.
$N_{\CT^{d+1}}$ in the second column denotes
the number of nodes of the corresponding symmetric cubature rules.
$\wt N_{\CT^{d+1}} $ in the third column denotes
the number of nodes of cubature rules which are obtained through Victoir's method.

We shall choose the cubature rule with fewest nodes for each $(d,m)$.
For example, in the case of $(d,m) = (12,6)$,
the cubature rule given in
Class 4.1.3 has $586$ nodes, whereas, the cubature given in Class 4.1.2
has $599$ nodes.
So in this case we choose the former cubature, and
do not list up the latter cubature in the table.

The cubature rules with $\wt N_{\CT^{d+1}} $ nodes in Table~\ref{tbl:1}
arise from a $4$-$(12,6,4)$, a $4$-$(13,6,12)$, a $4$-$(15,7,20)$,
a $4$-$(18,8,21)$, a $4$-$(11,5,1)$, a $4$-$(16,6,6)$.
Similarly,
we use a $4$-$(12,5,4)$, a $4$-$(14,7,20)$,
a $4$-$(15, 8, 40)$, a $4$-$(17, 8, 15)$, a $4$-$(16, 8, 60)$
to obtain cubature rules with $\wt N_{\CT^{d+1}} $ nodes in Table~\ref{tbl:2}.
Moreover we use the $5$-$(12,6,1)$ to obtain Table~\ref{tbl:3};
since nontrivial $5$-designs must have at least $12$ points,
we apply Victoir's method only for the case $(d,p,m) = (11,11,6)$.
See~\cite{GKL07} for the existence of $4$- and $5$-designs
that are used here.
To determine the value of $\wt N_{\CT^{d+1}} $,
we calculated the number of blocks of each design by using (\ref{eq:Tdes}).

\begin{table}[ht]
\caption{
Reducing the size of cubature of Class 4.1.2.
}
\label{tbl:1}
\begin{center}
\begin{tabular}{|c|c|c|}
\hline
$(d,m)$  &  $N_{\CT^{d+1}}$  &  $\wt N_{\CT^{d+1}} $ \\
\hline
$(11,6)$ & $1003$ & $211$ \\
\hline
$(12,6)$ & $1730$ & $586$ \\
\hline
$(14,7)$ & $6556$ & $901$ \\
\hline
\hline
$(10,5)$ & $485$ & $89$ \\
\hline
$(15,6)$ & $8041$ & $761$ \\ 
\hline
\end{tabular}
\end{center}
\end{table}
\begin{table}[ht]
\caption{
Reducing the size of cubature of Class 4.1.3.
}
\label{tbl:2}
\begin{center}
\begin{tabular}{|c|c|c|}
\hline
$(d,m)$  &  $N_{\CT^{d+1}}$  &  $\wt N_{\CT^{d+1}} $ \\
\hline
$(11,5)$ & $1025$ & $629$ \\
\hline
$(13,7)$ & $3811$ & $951$  \\
\hline
$(14,8)$ & $6906$ & $1251$  \\
\hline
$(16,8)$ & $25008$ & $1208$ \\
\hline
\hline
$(15,8)$ & $12903$ & $1593$ \\
\hline
$(17,8)$ & $43794$ & $954$ \\
\hline
\end{tabular}
\end{center}
\end{table}

\begin{table}[ht]
\caption{
Reducing the size of cubature of Class 4.2.
}
\label{tbl:3}
\begin{center}
\begin{tabular}{|c|c|c|}
\hline
$(d,p,m)$  &  $N_{\CT^{d+1}}$  &  $\wt N_{\CT^{d+1}} $ \\
\hline
$(11,11,6)$&  $1014$           & $222$ \\
\hline
\end{tabular}
\end{center}
\end{table}

\begin{exam}
We shall describe the $11$-dimensional cubature rule of Table~\ref{tbl:1}.
\begin{align*}
  & \tfrac{1}{(10)!} \int_{\CT^{11}} f(\xi) d\xi = 
       \tfrac{1331}{17472} f \left( (\tfrac{1}{11}, \ldots, \tfrac{1}{11} ) \right) + 
       \tfrac{1}{12012}
\sum_{x \in  \left( (1,0,\ldots,0) \right)  \CA_{10} } f(x) \\
     & \qquad
      + \tfrac{125}{24024} \sum_{x \in X} f \left( \tfrac{1}{5}  x \right) 
      + \tfrac{3375}{64064}
\sum_{x \in   \left( (\tfrac{1}{3}, \tfrac{1}{15}, \ldots, \tfrac{1}{15}) \right)  \CA_{10}} f(x),
\quad \forall f \in \CP_4^{11}.
\end{align*}
Here $X$ denotes the set of $66$ columns of a $11 \times 66$ incidence matrix
of the $4$-$(11, 5, 1)$ design.
Each element of $X$ is represented as a cyclic shift of the following six types
of vectors:
\begin{align*}
(1,1,1,1,0,0,0,0,0,1,0), \quad & (1,1,1,0,1,0,0,1,0,0,0), \quad  (1,1,1,0,0,1,1,0,0,0,0), \\
(1,1,0,1,1,0,0,0,1,0,0), \quad & (1,1,0,1,0,1,0,1,0,0,0), \quad  (1,1,0,0,1,0,1,0,0,1,0).
\end{align*}
\end{exam}

\begin{rem}
The above $4$-$(v, k, \lambda)$ designs are chosen so that
$\lambda$ is the minimum for given $v, k$.
For example, when $(v, k) = (12, 6)$,
Eq.(\ref{eq:Tdes}) implies $\lambda \equiv 0 \pmod 2$.
The nonexistence of $4$-$(12, 6, 2)$
was shown by B.~D.~McKay and S.~P.~Radziszowski in 1996.
The existence of $4$-$(12, 6, 4)$ is known and hence
a $4$-$(12, 6, 4)$ is the minimum design.
Some nonexistence criteria are derived from
equations of block intersection numbers by E.~K\"ohler (1985)
and J.~Bolick (1990), and
others from a linear programming bound by Delsarte (1973).
The authors knew these informations
through discussion with Reinhard Laue.
Of course, a $5$-$(12,6,1)$ is the minimum design.
\end{rem}

\section{Positive symmetric cubature rules of degree $9$ on the sphere}
\setcounter{equation}{0} 

In this section we construct cubature rules of index 8 on the sphere $\SS^d$, or degree 9 
on $\SS^{d-1}$, that are invariant under the group $\CB_{d+1}$.
According to Propositions~\ref{prop:cuba-SSd0} and \ref{prop:cuba-CTd},
this is equivalent to finding symmetric cubature rules on the simplex for the integral 
$$
      w_{-\frac{1}{2}} \int_{\CT^{d+1}} f(\xi)  \frac{d\xi}{\sqrt{\xi_1 \cdots \xi_{d+1}}} =  
                w_{-\frac{1}{2}} \int_{T^d} f(x)  \frac{dx}{\sqrt{x_1 \cdots x_d (1-|x|_1)}}.
$$
The moments for this integral is given by \eqref{moment-g} with $\g = (-\frac12,\ldots, -\frac12)$,
\begin{equation}\label{moment-1/2}
  m_\a^{(-\frac12)} =  \frac{ (\frac12)_{\a_1}\cdots (\frac12)_{\a_{d+1} }}
{(\frac{d+1}{2})_{|\a|}}, \qquad \a \in \NN_0^d.  
\end{equation}
The integrals of $S_{\{\ell\}}$ for partitions of $4$ are given by 
\begin{align*}
   m_{\{4\}} & = \frac{105}{16 (\frac{d+1}{2})_4} (d+1), \quad
        m_{\{2,2\}} = \frac{9}{16(\frac{d+1}{2})_4} \binom{d+1}{2}, \\ 
   m_{\{2,1,1\}} & = \frac{9}{16 (\frac{d+1}{2})_4} \binom{d+1}{3}, \quad 
        m_{\{1,1,1,1\}} = \frac{1}{16 (\frac{d+1}{2})_4} \binom{d+1}{4}. 
\end{align*}

We report those results in cases depending on the values of $(k_1,k_2,k_3,k_4)$.
In order to have cubature rules with fewer nodes, we shall choose $k_1 = d+1$ whenever possible 
and choose $k_1,k_2,k_3,k_4$ as small as possible otherwise. The values of $p, m, a$ are parameters
and will be chosen so that the cubature rules are positive. All cubature rules are reported as
cubature rules of index type on the sphere, because of their connection with other topics. The results 
that we report below are mostly of the type 
$$
         (k_1,k_2,k_3,k_4)  =  (d+1,1,2,m), \quad m = 3,4, ... . 
$$
In other word, they are of the form
\begin{align} \label{CF-case1}
&   \s_d    \int_{\SS^d} f(x) d\s(x)    =
\frac{\l_1}{2^{d+1}}  \sum_{x \in
\left( \tfrac{1}{\sqrt{d+1}}, \ldots,\tfrac{1}{ \sqrt{d+1}}\right)
\CB_{d+1}} f(x)  +  \frac{\l_2}{2(d+1)}
\sum_{x \in
(1, 0, \cdots, 0)
\CB_{d+1} } f(x)  \notag \\
& \quad
+ \frac{\l_3}{ 4\binom{d+1}{2}}
\sum_{x \in  
\left (\tfrac{1}{\sqrt{2}}, \tfrac{1}{\sqrt{2}}, 0,\ldots,0\right)
\CB_{d+1} } f(x)
        + \frac{\l_4}{ 2^m \binom{d+1}{m}}
\sum_{x \in  
\left( \tfrac{1}{\sqrt{m}}, \ldots,  \tfrac{1}{\sqrt{m}}, 0,\ldots,0 \right)
\CB_{d+1}} f(x)
         \notag  \\ 
& \quad
+ \frac{\l_5}{2^{p+1}(d+1)\binom{d}{p}}
\sum_{x \in
\left(\sqrt{\tfrac{a}{a+p}},\tfrac{1}{\sqrt{a+p}},  \ldots, \tfrac{1}{\sqrt{a+p}}, 0,\ldots,0\right)
\CB_{d+1}} f(x),  \notag 
\end{align}
where we have normalized the nodes so that they are points on the sphere and normalized our 
cubature coefficients by 
$$
   \l_1 = A_1 (d+1)^4, \quad \l_2 = A_2,  \quad \l_3 = 2^4 A_3, \quad \l_4 = m^4 A_4, \quad \l_5 = (a+p)^4 A_5,
$$
where $A_i$ satisfy \eqref{system}, so that $\sum_{k=1}^5 \l_k =1$, from which the cubature rules of degree type 
on the sphere can be deduced easily. 

The solutions are obtained in the same manner as in Section 4. Hence, we shall state our results
without further discussion on how they are obtained. Again, the value of $a$ that gives positive cubature 
rules is not unique. We often choose $a$ so that one of the coefficient is zero to reduce the number of points.
It is also interesting to choose $a$ as a square of an integer, which would lead to a rational solution
that can be used to write $\langle x,x \ra^4$ as a sum of $8$th powers of linear forms with rational coefficients,
as discussed in Subsection 2.3. 

For each cubature rule, we give the value of $a$ and the parameters $d, p ,m$, but we will not give the 
coefficients $\l_1,\ldots,\l_5$, which can be determined from solving the linear system of equations \eqref{system} 
using the given parameters. For cubature rules of the same dimension, we report only rules that have either smallest 
number of nodes or have some other feature, such as a rational solution. 

\subsection{Case $(k_1,k_2,k_3,k_4) = (d+1,1,2,m)$ and $p =d$}
For this case we see, after many trails, that the system \eqref{system} has a positive solution in two classes 
of parameters, one with $\l_2 =0$, the other with $\l_3 =0$. 

\medskip\noindent
5.1.1. {\bf  $\l_2 =0$}. Positive cubature rules hold for 
$$
     3 \le d \le 17, \quad \hbox{and} \quad d = 19, 20, 23, 24,  28, 32,
$$
with the number of nodes 
$$
         N_{\SS^d} =   4 \binom{d+1}{2} + 2^m \binom{d+1}{m} + 2^{d+1}(d+2),
$$
or equivalently, $N_{\CT^{d+1}} =   d+2 + \binom{d+1}{2}+ \binom{d+1}{m}$
in terms of cubature rules on the simplex. The solution is given with 
\begin{align*}
&    a  =  \frac{(-43 + d^2 + 45 m -  3 d m + \sqrt{\Delta})} {6(m-1)},  \quad \hbox{with}  \\
&     \Delta = -4 (-3 + 3 m) (-6 d + 3 d^2 + 6 m - 3 d m) + (43 - d^2 - 45 m + 3 d m)^2.
\end{align*}
The positive cubature holds for $(d,m)$ in the following pairs: 
\begin{align*}
    &  m = 3: \, 3 \le d \le 5, \quad m =4:  \, 5 \le d \le 9, \quad m = 5: \, 9 \le d \le 14, \\
     & m  = 6: \, 13 \le d \le 17, \quad
      m = 7:  \,  d = 19, 20, \quad m = 8: \, d = 23, 24, \\
      & m=9: \, d = 28, \quad m = 10:\, d = 32.  
\end{align*}
There are no positive solutions for $m > 10$. 

\medskip\noindent
5.1.2. {\bf  $\l_3 =0$}. Positive cubature rules hold for 
$$
    d = 6, 7, 10, 14, 17, 18, 21, 22, 25, 26, 29, 33, 36, 37, 40, 44.
$$
Furthermore, the solutions are divided into three groups: 
$$
       d = 4m -8, \quad 4m -7, \quad 4m-6. 
$$ 
The number of nodes of the cubature rule on the sphere is equal to 
$$
         N_{\SS^d} = 2 (d+1) + 2^m \binom{d+1}{m} + 2^{d+1}(d+2),
$$
or equivalently, $N_{\CT^{d+1}} =  2d+3 + \binom{d+1}{m}$ in terms of cubature rules on the simplex.

\medskip
{\it Case $d = 4m-6$}.  We have positive cubature rules if $3 \le m \le 8$ and $a = 9$. The parameters are
$$
  (d,m) = (6, 3), (10, 4), (14,5), (18,6), (22, 7), (26,8).
$$

\medskip
{\it Case $d = 4m-7$}.  The cubature rules are positive if $6 \le m \le 11$ and 
\begin{align*}
  a = \frac{3(3m-7)}{m-3},  
\end{align*}
The parameters are
$$
  (d,m) = (17,6), (21, 7), (25,8), (29, 9), (33, 10), (37,11).
$$

\medskip
{\it Case $d = 4m-8$}.  The cubature rules are positive if
$   m = 11, 12, 13$ and 
\begin{align*}
      a = \frac{3(3m-8)}{m-4}.  
\end{align*}
The parameters are
$$
  (d,m) = (36, 11), (40, 12),  (44,13).
$$

\subsection{Case $(k_1,k_2,k_3,k_4) = (d+1,1, 2, m)$, $p = d$ and $a = 9$}
The cubature rules are positive for $3 \le d \le 28$. The parameters are 
\begin{align*}
 & m=3: \, 3 \le d \le 6, \quad m=4: \, 7 \le d \le 10, \quad m=5:\, 11 \le d \le 14, \\
 & m=6: \, 15 \le d \le 18, \quad m = 7: \, 19 \le d \le 22, \quad m=8:\, 13 \le d \le 26.
\end{align*}
The cases $d = 6, 10, 14, 18, 22, 26$ have already appeared in Case $d = 4m-6$ of
Subsection 5.1.2, in which $\l_3 =0$.  In all other cases, none of $\l_i$ is zero and the number of nodes 
of the cubature  rule on the sphere is equal to 
$$
         N_{\SS^d} = 2 (d+1) + 4 \binom{d+1}{2} + 2^m \binom{d+1}{m} + 2^{d+1}(d+2),
$$
or equivalently, $N_{\CT^{d+1}} =  2d+3 + \binom{d+1}{2}+ \binom{d+1}{m}$ in terms of cubature rules on the simplex,
which are more than the number of nodes for cubature rules in  Subsection 5.1. However, the fact that $a = 3^2$ 
means that we obtain rational solutions for the cubature rules on the sphere.
Therefore, as noted in Subsection 2.3, we obtain a representation of $(\sum_{i=1}^{d+1} x_i^2)^4 $ as a sum of $8$th 
powers of linear forms with rational coefficients. Those identities are not only of interesting itself, but
played a key role in Hilbert's proof of the Waring problem. See~\cite[p.~722]{Dickson} for details.

\subsection{Case $(k_1,k_2,k_3,k_4) = (d+1,1,2,m)$ and $2 \le p \le d-1$ }

The positive cubature rules are found mostly with $p = m-1, m, m+1$. We consider two cases. 

\medskip\noindent
5.3.1. $p = m-1$. Here we found positive cubature rules with 
\begin{align*}
 &a= \frac{1}{6d} \left( -49 - 6 d + d^2 + 51 m - 3 d m + \sqrt{\Delta}\right)  \\
   & \qquad\quad\hbox{with} \quad  \Delta = 36 d (2 + d - m) (-3 + m) + (-49 + d^2 + 51 m - 3 d (2 + m))^2,
\end{align*}
for which $\l_2 = 0$, and the parameters are 
\begin{align*}
   (d,m) = & (3, 3), (4, 3), (5, 4), (6, 4), (7, 5), (8, 5), (10, 6), (12, 7), (13, 7), \\
                 & (14, 7), (15, 8), (16, 8), (17, 9), (18, 9), (19, 9), (20, 9).
\end{align*}
The number of nodes of such a cubature of dimension $d$ is 
$$
   N_{\SS^d} = 2^{d+1} + 4 \binom{d+1}{2} + 2^m \binom{d+1}{m} + 2^{m}(d+1)\binom{d}{m-1},
$$
which however is often larger than the number of points of the corresponding cubature in
Subsection 5.1. 

5.3.2. $a = 4$. In this case we obtain positive cubature rules for the parameters
\begin{align*}
  (d,m,p)  = & (3,3,2), ( 4, 3, 2 ),  (5 ,4 ,3 ),  (6 ,4 ,3 ),  (7 ,5 ,4 ),  (8 ,4 ,4 ),  (9 ,4 ,5 ), \\
    &   (10 ,4 ,5 ),  (11 ,5 ,6 ),  (12 ,6 ,6 ),  (13 ,6 ,7 ),  (14 , 6,7 ),  (15 , 6,7 ). 
\end{align*}
The number of nodes of such a cubature of dimension $d$ is 
$$
   N_{\SS^d} = 2 (d+1) + 2^{d+1} + 4 \binom{d+1}{2} + 2^m \binom{d+1}{m} + 2^{m}(d+1)\binom{d}{m-1}.
$$
The choice of $a = 2^3$ means that all these cubature rules have rational coefficients. The case of $(3,3,2)$ 
is of particular interesting, as it corresponds to a formula of Hurwitz (see below) that represents $(x_1^2+x_2^2+x_3^2+x_4^2)^4$
as a sum of
8th powers of real linear forms. We give the coefficients of this case explicitly.

\medskip\noindent
{\it Case} $(d,m,p) = (3,3,2)$. 
$$
   a = 4, \quad \l_1  = \frac{2} {15}, \quad \l_2  = \frac{1} {15}, \quad \l_3  = \frac{1} {8}, \quad \l_4  = 0, \quad \l_5 = \frac{27} {40}. 
$$

\begin{rem}
\label{rem:Hurwitz}
(i) Using \eqref{eq:iso}, our solution in Case $(d,m,p) = (3,3,2)$ above leads to the identity 
\begin{align*}
5040 \Big(\sum_{i=1}^4 x_i^2\Big)^4 & = 6\sum_4 (2x_i)^8
+ 60 \sum_{12} (x_i \pm x_j)^8 \nonumber \\
& \qquad + \sum_{48} (2 x_i \pm x_j \pm x_k)^8
+ 6  \sum_8 (x_1 \pm x_2 \pm x_3 \pm x_4)^8.
\end{align*}
This identity was found by A.~Hurwitz (cf.~\cite[p.~721]{Dickson}) in connection with the Waring
problem. We recall a $144$-node cubature rule of degree $9$ on $\SS^3$
given in Case $(m, d) = (3, 3)$ of Subsection 5.3.
This is centrally symmetric, and so
can be reduced by half by Corollary~\ref{cor:cuba-SSd1}.
The resulting $72$-node cubature rule of index $8$ gives the identity
\begin{align*}
315 \Big(\sum_{i=1}^4 (2x_i)^2\Big)^4 & = 60 \sum_8 (x_1 \pm x_2 \pm x_3 \pm x_4)^8
+ 60 \sum_4 (2x_i)^8 + 6 \sum_{12} (2x_i \pm 2x_j)^8 \nonumber \\
& \quad + \sum_{16} (2x_i \pm 2x_j \pm 2x_k)^8
+ \sum_{32} (3x_i \pm x_j \pm x_k \pm x_l)^8,
\end{align*}
which implies that the 4th power of a sum of four squares of even integers, multiplied by $315$,
is a sum of at most $840$ 8th powers of integers. Hurwitz's identity implies that
the 4th power of a sum of four squares, multiplied by $5040$,
is a sum of at most $840$ 8th powers of integers.
Without requiring rationality, there are some cubature rules with fewer nodes.
Salihov~\cite{Salikov} proved that the $120$ vertices of the regular polytope in $\RR^4$
give a cubature rule of degree $11$ on $\SS^3$.
By Corollary~\ref{cor:cuba-SSd1},
this is reduced to a cubature of index $10$ with $60$ nodes, which
can be viewed as an index-eight cubature rule.  
On the other hand, the $4$-dimensional cubature rule of degree $9$ given in Subsection~5.1
has $136$ nodes. This cubature is centrally symmetric and so
can be reduced to a $67$-point cubature rule of index $8$.
We do not know other examples of index-eight cubature rules on $\SS^3$
with fewer nodes than Hurwitz's rule.
Few explicit constructions of spherical cubature rules
with few nodes are known even for small $d$.

(ii) 
Shatalov~\cite[p.~179]{Shatalov} found
a cubature rule of index $8$ on $\SS^4$ with $300$ nodes.
To obtain this cubature rule,
she developed a recursive construction~\cite[Theorem~4.4.9]{Shatalov}
to push up the dimension of a given cubature rule.
The $300$-node cubature rule comes from Salihov's cubature rule on $\SS^3$ with $60$ nodes.
The $5$-dimensional cubature rule given in Subsection 5.1 
has $312$ nodes, which is however centrally symmetric and, therefore, can be reduced to a 
$156$-node cubature rule by Corollary~\ref{cor:cuba-SSd1}. Thus, our cubature rule has 
fewer nodes than Shatalov's cubature rule.
By Corollary~\ref{cor:cuba-SSd1},
the cubature rule on $\SS^4$ given in
Subsection 5.2
can be reduced a $161$-node cubature rule, which
is also smaller than Shatalov's cubature.
This cubature rule also has rational coefficients and nodes and gives the identity
\begin{align*}
&
161280 \Big( \sum_{i=1}^5 x_i^2 \Big)^4
\\
& =
59 \sum_{16} (x_1 \pm x_2 \pm x_3 \pm x_4 \pm x_5)^8
+ 52 \sum_5 (2x_i)^8
+ 16 \sum_{20} (2 x_i \pm 2 x_j)^8 \nonumber \\
& \qquad + 4 \sum_{40} (2 x_i \pm 2 x_j \pm 2 x_k)^8
+ 28561 \sum_{80} (3 x_1 \pm x_2 \pm x_3 \pm x_4 \pm x_5)^8.
\end{align*}
We also note that
the $5$-dimensional cubature rules given in Subsection 5.3
can be reduced to a $196$-node and a $201$-node cubature rule, which
are smaller than Shatalov's rule.

(iii) Lyubich and Vaserstein~\cite[Theorem~2.8]{Lyubich-Vaserstein} proved that
there exists a cubature rule of index $2t$ on $\SS^d$ with at most
$\dim \CP_{2t}^{d+1}(\SS^d) = \binom{d+2t}{d}$ nodes.
Applying Victoir's criterion to the cubature rules given in this section
produces cubature rules with fewer nodes than
the cubature of Lyubich and Vaserstein.
For example, let us consider the case $d = 14$.
It is well known (cf.~\cite[Table~12.1]{HSS99}) that
there is an $OA(2^{12}, 14, 2, 8)$, say $A$.
This OA has antipodality, namely,
if $x$ is a row of $A$, then so is $-x$.
This is because the OA is constructed from a cyclic linear code over $\mathbb{F}_2$
by replacing $0, 1 \in \mathbb{F}_2$ by $\pm 1$.
We can easily construct an $OA(2^{13}, 15, 2, 8)$ by defining
the $2^{13} \times 15$ array to be
\begin{equation*}
\left(\begin{array}{cc}
A & {\bf 1} \\
A & {\bf -1} \\
\end{array} \right).
\end{equation*}
This OA is again antipodal.
The existence of a simple $4$-$(15, 5, 2)$ design, namely,
a $4$-$(15, 5, 2)$ design without repeated blocks, is known.
This information is not covered in~\cite{GKL07}.
The authors knew this information through private communication with Reinhard Laue.
By (\ref{eq:Tdes}), the number of blocks is equal to $546$.
Thus, applying Proposition~\ref{prop:cuba-Victoir} to the
$OA(2^{13}, 15, 2, 8)$ and $4$-$(15, 5, 2)$ design,
the $620414$-node cubature given in Class 5.1.2 can be reduced to
a cubature rule of degree $9$ on $\SS^{14}$ with $148574$ nodes, or
a cubature rule of index $8$ on $\SS^{14}$ with $74287$ nodes. 
When $d = 14$, the cubature of Lyubich and Vaserstein has $319770$ nodes.
Other $15$-dimensional spherical cubature rules can be constructed by
applying Shatalov's recursive construction
to a $14$-dimensional cubature rule with $89964$ nodes found by de la Harpe et al.~\cite{HPV07}.
The nodes of Harpe's cubature rule arise from shells of
the Quebbemann lattice, and therefore are antipodal.
Hence by Corollary~\ref{cor:cuba-SSd1},
this cubature rule can be reduced to a half-node cubature, to which
we apply Shatalov's recursion.
The resulting cubature has $224910$ nodes, which
is bigger than our $74287$-node cubature rule.
We do not known other explicit constructions of cubature rules of index $8$ on $\SS^{14}$
with at most $74287$ nodes.

It would be interesting to study in detail the geometric configuration of
points of cubature given in this section, and
compare with cubature obtained from the lattice approach.
This will be discussed in a forthcoming paper.
\end{rem}

\section{Positive symmetric cubature rules of degree $11$ on the sphere}
\setcounter{equation}{0} 
Using the moments given in \eqref{moment-1/2}, the quantities of $a$ and $A_6$ in
\eqref{aA6} are given by 
\begin{equation} \label{aA6t=5}
            a = 10 \frac{p-1}{d-1} -1, \qquad 
            A_6 = \frac{60 d}{p a (a-1)^2(a+1)(d+9)(d+7)(d+5)(d+3)}.
\end{equation}
The results that we report below are mostly of the type 
$$
         (k_1,k_2,k_3,k_4,k_5)  =  (d+1,1,2,3,m), \quad m = 3,4, ... . 
$$
In other word, they are of the form
\begin{align} \label{CF-degree11}
&
   \s_d    \int_{\SS^d} f(x) d\s(x)  =
\frac{\l_1}{2^{d+1}}  \sum_{x \in 
\left(\tfrac{1}{\sqrt{d+1}}, \ldots,\tfrac{1}{ \sqrt{d+1}}\right)
\CB_{d+1}} f(x)  +  \frac{\l_2}{2(d+1)}
\sum_{x \in
(1,0,\cdots,0)
\CB_{d+1} } f(x)  \notag \\
&
\qquad
+ \frac{\l_3}{ 4\binom{d+1}{2}} \sum_{x \in 
\left (\tfrac{1}{\sqrt{2}}, \tfrac{1}{\sqrt{2}}, 0,\ldots,0\right)
\CB_{d+1} } f(x)
        + \frac{\l_4}{ 8 \binom{d+1}{3}}  \sum_{x \in 
\left (\tfrac{1}{\sqrt{3}}, \tfrac{1}{\sqrt{3}},  \tfrac{1}{\sqrt{3}}, 0,\ldots,0\right)
\CB_{d+1}} f(x)
         \notag  \\ 
&
\qquad
+ \frac{\l_5}{ 2^m \binom{d+1}{m}} \sum_{x \in 
\left( \tfrac{1}{\sqrt{m}}, \ldots,  \tfrac{1}{\sqrt{m}}, 0,\ldots,0 \right)
\CB_{d+1}} f(x)
\notag \\
& \qquad
    + \frac{\l_6}{2^{p+1}(d+1)\binom{d}{p}} \sum_{x \in 
\left( \sqrt{\tfrac{a}{a+p}},\tfrac{1}{\sqrt{a+p}},  \ldots, \tfrac{1}{\sqrt{a+p}}, 0,\ldots,0\right)
\CB_{d+1}} f(x),  \notag 
\end{align}
where we have normalized the nodes so that they are points on the sphere and normalized our 
cubature coefficients by 
$$
  \l_1 = A_1 (d+1)^5, \, \l_2 = A_2, \,  \l_3 = 2^5 A_3, \,   \l_4 = 3^5 A_4, \, \l_5 = m^5 A_5, \, \l_5 = (a+p)^5 A_6
$$
and
$A_i$ satisfy \eqref{system5}, so that $\sum_{k=1}^6 \l_k =1$, from which the cubature rules of degree type 
on the sphere can be deduced easily.

The case of $d = 2p -1$ is of particular interest, since in this case $a = 4$ and the solution is rational. 

\medskip\noindent
{\bf Case $(d+1,1,2,3,m)$}.
We found positive cubature rules for $3 \le d \le 23$.  We specify the parameters $d, p , m$. The solutions turned out 
to be all around $d = 2 p -1$. Those with odd dimensions are 
\begin{align*}
  (d,p,m) =& (3,2,4), (5,3,4),(7,4,4), (9, 5, 5), (11, 6, 5), (13, 7, 5),\\ 
                 & (15, 8, 5), (17, 9, 5), (19, 10, 5), (21, 11, 5),  (23,12,4),
\end{align*}
which are all rational solutions. Those with even dimensions are 
\begin{align*}
   (d, p, m) =& (4,3,4),  (6,4,4), (8, 5, 5), (10, 6, 5), (12, 7, 5),\\
                    & (14, 8, 6), (16, 9, 6), (18, 10, 6), (20, 11, 6), (22, 12, 6). 
\end{align*}

Among these solutions, we single out the following ones because they have at least one coefficient zero. In each
case we give the coefficients explicitly. 

\medskip\noindent
{\it Case $(d,p,m) = (3,2,4)$}. 
$$
 \l_1 = \frac{2}{15}, \quad   \l_2 =\frac{1}{15},  \quad   \l_3 =\frac{1}{8},  \quad  \l_4 = 0,  \quad   \l_5 =0,  \quad 
    \l_6 = \frac{27}{40}.
$$

\medskip\noindent
{\it Case $(d,p,m) = (7,4,4)$}. 
$$
 \l_1 = \frac{32}{315}, \quad   \l_2 =\frac{1}{105},  \quad   \l_3 =\frac{2}{45},  \quad  \l_4 = 0,  \quad   \l_5 = \frac{2}{15},  \quad 
    \l_6 = \frac{32}{45}.
$$

\medskip\noindent
{\it Case $(d,p,m) = (21,11,5)$}. 
$$
 \l_1 = \frac{14641}{705432}, \,   \l_2 =\frac{61}{192192},  \,  \l_3 =0,   \,  \l_4 = \frac{405}{31616},
   \,   \l_5 = \frac{2625}{1414415},  \, 
    \l_6 = \frac{16875}{18304}.
$$

\medskip\noindent
{\it Case $(d,p,m) = (23,12,4)$}. This is a particularly interesting case, since three coefficients are zero. 
$$
 \l_1 =0, \quad   \l_2 =\frac{1}{4095},  \quad   \l_3 =0,  \quad  \l_4 = 0,  \quad   \l_5 = \frac{506}{12285},  \quad 
    \l_6 = \frac{11776}{12285}.
$$

\begin{rem}
(i) By \eqref{eq:iso}, our solution in the case $(d,m,p) = (3,2,4)$ above leads to the identity 
\begin{align*}
22680
\Big( \sum_4 x_i^2 \Big)^5
&= 9 \sum_4 (2x_i)^{10} + 180 \sum_{12} (x_i \pm x_j)^{10} + \sum_{48} (2x_i \pm x_j \pm x_k)^{10} \\
& \qquad + \sum_{48} (2x_i \pm x_j \pm x_k)^{10} + 9 \sum_8 (x_1 \pm x_2 \pm x_3 \pm x_4)^{10}.
\end{align*}
This identity was found by I.~Schur~\cite[p.~721]{Dickson} in connection with the Waring problem.
It is interesting to note that the coefficients of the cubature rule of the case $(d,p,m) = (3,2,4)$
are exactly the same as the case $(d,p,m) = (3,3,2)$ in Section 6.
Namely, Schur's identity implies Hurwitz's identity, though
they are different from each other as representations; see Remark~\ref{rem:Hurwitz} (i).
Schur's identity was given in a paper by Landau in 1908 who
did not mention its connection to Hurwitz's identity. 

(ii) Harpe et al.~\cite{HPV07} constructed
several even-dimensional cubature rules of index $10$ on spheres,
using extremal $\ell$-modular lattices.
Under some conditions,
shells of such lattices naturally give
spherical cubature rules with special degrees.
For example,
according to Bachoc-Venkov~\cite[Corollat 4.1]{Bachoc-Venkov},
shells of extremal $3$-modular lattices in $\RR^d$ form a cubature rule of degree at least $5$
if $d \equiv 0, 2 \pmod{12}$.
Harpe et al. did not discuss
the explicit constructions of cubature rules of index $10$
in odd dimensions, as well as in some even dimensions up to $24$.
Whereas, in this section,
we directly constructed several cubature rules
in those dimensional spaces (e.x. $\RR^{18}$).
Of course,  the lattice approach may still work well in those dimensional spaces.
For example, 
collecting suitable shells of an extremal $3$-modular lattice in $\RR^{18}$
(such a lattice actually exists), we may still be able to obtain cubature rules of index $10$ on $\SS^{17}$.
\end{rem}


\begin{thebibliography}{99}
\bibitem{Salikov}
     G. N. Salihov, 
     Cubature formulas for a hypersphere that are invariant with respect to the group
     of the regular $600$-face, 
     \textit{Soviet Math. Dokl}., \textbf{16} (1975), 1046-1050.
\end{thebibliography}
\end{document}